\documentclass[11pt]{amsart}

\usepackage{amsfonts}
\usepackage{amsmath}
\usepackage{amssymb}
\usepackage{amsthm}
\usepackage{amsxtra}
\usepackage{epic, eepic}
\usepackage{epsfig,color}
\usepackage{fullpage}
\usepackage{vmargin}
\theoremstyle{plain}
\newtheorem{Thm}{Theorem}[subsection]
\newtheorem{Cor}[Thm]{Corollary}
\newtheorem{Lem}[Thm]{Lemma}
\newtheorem{Prop}[Thm]{Proposition}

\theoremstyle{definition}
\newtheorem*{Def}{Definition}
\newtheorem{Ques}{Question}

\newtheorem*{Conve}{Convention}

\theoremstyle{remark}
\newtheorem{Rem}{Remark}
\newtheorem*{Remrk}{Remark}
\newtheorem{Ex}{Example}[subsection]

\setmarginsrb{33mm}{26mm}{33mm}{40mm}%
            {0mm}{10mm}{0mm}{10mm}

\begin{document}

\title{On isoperimetric profiles of algebras}

\author{Michele D'Adderio}\address{Department of Mathematics\\
University of California, San Diego\\
9500 Gilman Drive \#0112\\
La Jolla, CA 92093-0112\\
United States of America (USA)}\email{mdadderio@math.ucsd.edu}

\begin{abstract}
Isoperimetric profile in algebras was first introduced by Gromov in
\cite{gromov}. We study the behavior of the isoperimetric profile
under various ring theoretic constructions and its relation with the
Gelfand-Kirillov dimension.
\end{abstract}

\maketitle

\section*{Introduction}

The geometric concept of an isoperimetric profile was first
introduced in algebra for groups by Vershik in \cite{vershik} and
Gromov in \cite{gromovAsyInv}. Here is the definition given by
Gromov in \cite{gromov}, for semigroups:
\begin{Def}
Given an infinite semigroup $\Gamma$ generated by a finite subset $S$,
and given a finite subset $\Omega$ of $\Gamma$ we define the
\textit{boundary} of $\Omega$ as
$$
\partial_S(\Omega):=\bigcup_{s\in S}(s\Omega\setminus \Omega).
$$

Then we define the \textit{isoperimetric profile of a semigroup}
$\Gamma$ with respect to $S$ as the function from $\mathbb{N}$
onto itself given by
$$
I_{\circ}(n;\Gamma,S):=\inf_{|\Omega|=n} |\partial_S(\Omega)|
$$
for each $n\in \mathbb{N}$, where $|X |$ denotes the cardinality
of the set $X$.
\end{Def}
It's well known that the asymptotic behavior of this function is
independent of the set of generators $S$.

For properties of the isoperimetric profile see
\cite{erschler,erschler1,gromov,pittet}, the survey
\cite{pittetcoste} and references therein.

The notion of isoperimetric profile for algebras was introduced by
Gromov in \cite{gromov}:

\begin{Def}
Let $A$ be a finitely generated algebra over a field $K$ of
characteristic zero. Given two subspaces $V$ and $W$ of $A$ we
define the \textit{boundary} of $W$ with respect to $V$ by
$$
\partial_V(W):=VW/(VW\cap W).
$$

If $V$ is a generating finite dimensional subspace of $A$, we
define the \textit{isoperimetric profile of $A$ with respect to
$V$} to be the maximal function $I_{*}$ such that all finite
dimensional subspaces $W\subset A$ satisfy the following
\textit{isoperimetric inequality}
$$
I_{*}(|W|;A,V)=I_{*}(|W|)\leq |\partial_V(W)|,
$$
where $|Z|$ denotes the dimension over the base field $K$ of the
vector space $Z$.
\end{Def}
\vfill
\noindent\line(1,0){100}\\
{\small \textit{Keywords}: isoperimetric profile, Gelfand Kirillov
dimension, amenability.}%

Again, the asymptotic behavior of this function does not depend on
the generating subspace.

In \cite{gromov} Gromov studied in particular the isoperimetric
profile of group algebras and its relation with the isoperimetric
profile of the underlying group.

Unless otherwise stated, we consider finitely generated algebras
over a field of characteristic zero.

The isoperimetric profile is an asymptotically weakly sublinear
function, and it's linear if and only if the algebra is
nonamenable (in the sense of Elek \cite{elek1}). In this sense it
can be viewed as a measure of the amenability of an algebra.

We start by studying the isoperimetric profiles of some related
algebras. The main results can be stated as follows in the case of
finitely generated algebras:
\begin{Thm} \label{theorem1}
The isoperimetric profile of a finitely generated algebra $A$ is
asymptotically equivalent to the isoperimetric profile of a (right)
localization of $A$ with respect to a right Ore subset of regular
elements.
\end{Thm}
\begin{Remrk}
We will give later in the paper a precise definition on what we mean by asymptotical equivalence. Moreover, in the previous theorem, considering
localizations, we may get an algebra that is not finitely generated. Indeed we will see that in these cases it will make sense to talk about the
isoperimetric profile of these algebras, and the statements will turn out to be precise.
\end{Remrk}
\begin{Thm} \label{theorem2}
If the associated graded algebra $gr(A)$ of a filtered finitely
generated algebra $A$ is a finitely generated domain, then the
isoperimetric profile of $A$ is asymptotically (weakly) faster then
the isoperimetric profile of $gr(A)$.
\end{Thm}

The following theorem generalizes some of the results in
\cite{elek2} about division algebras. We state it here in the case
of finitely generated algebras.
\begin{Thm} \label{theorem3}
If $B\subset A$ are finitely generated domains and $B$ is right Ore,
then the isoperimetric profile of $B$ is asymptotically (weakly)
slower than the isoperimetric profile of $A$.
\end{Thm}

Given an amenable domain, it's not true that a subdomain must be
amenable. In fact it's well known that the Weyl algebra $A_1$ is
amenable, since it has finite $GK$-dimension, hence by \cite{elek2}
(or even by Theorem \ref{theorem1}) its quotient division algebra
$D_1$ is still amenable. But it's also known (see \cite{makar}) that
$D_1$ contains a subalgebra isomorphic to a free algebra of rank
$2$, which is known to be nonamenable. One of the main result of the
paper is that this is the only case that can occur:
\begin{Thm} \label{theorem4}
If $A$ is an amenable domain which does not contain a subalgebra
isomorphic to a noncommutative free algebra, then all the subdomains
of $A$ are amenable.
\end{Thm}

We computed the isoperimetric profile of various algebras:
\begin{Thm} \label{theorem5}
The isoperimetric profile of the following algebras is of the form
$n^{\frac{d-1}{d}}$ where $d$ is the $GK$-dimension of the algebra:
\begin{itemize}
    \item finitely generated algebras of $GK$-dimension $1$,
    \item finitely generated commutative domains,
    \item finitely generated prime $PI$ algebras,
    \item universal enveloping algebras of finite dimensional Lie
    algebras,
    \item Weyl algebras,
    \item quantum skew polynomial algebras,
    \item quantum matrix algebras,
    \item quantum groups $GL_{q,p_{ij}}(d)$,
    \item quantum Weyl algebras,
    \item quantum groups $\mathcal{U}(\frak{sl}_2)$ and
    $\mathcal{U}'(\frak{sl}_2)$.
\end{itemize}
\end{Thm}
Notice that not all algebras have an isoperimetric profile of this form. In section 3.1 there is an
example which is due to Jason Bell of an algebra of $GK$-dimension $2$ but with constant
isoperimetric profile.

We will also study the relation of the isoperimetric profile with
other invariants for algebras. In particular we will discuss the
relation between the isoperimetric profile and the lower
transcendence degree introduced by Zhang in \cite{zhang}. This is
a non negative real number (or infinity) associated to an algebra
$A$ (we will give the definition later), denoted by
$\mathrm{Ld}A$, with the property that
$$
\mathrm{Ld}A\leq \mathrm{Tdeg}A\leq GK\dim A,
$$
where $\mathrm{Tdeg}A$ denotes the $GK$-transcendence degree of
$A$ (see \cite{zhang1} for the definition). In section 4.2 we show
that the isoperimetric profile is a finer invariant than the lower
transcendence degree, and we use it to answer a question in
\cite{zhang}:
\begin{Prop}
The group algebra $K\Gamma$ of an ordered semigroup $\Gamma$ is
$\mathrm{Ld}$-stable, i.e. $\mathrm{Ld}K\Gamma=GK\dim K\Gamma$.
\end{Prop}

This connection allows us also to provide new examples of amenable
domains and division algebras with infinite $GK$-transcendence
degree (cf. \cite{elek2}).

In the last section of the paper we answer a question by Gromov in
\cite{gromov} Section 1.9.

The paper is divided into four sections which are organized as
follows:
\begin{itemize}
    \item In the first
section we provide definitions and basic properties of the
isoperimetric profile, particularly its connection with the notion
of amenability.
    \item In the second section we study the behavior of the isoperimetric profile
    under various ring-theoretic constructions. We will consider subalgebras, homomorphic images,
localizations, modules over subalgebras, tensor products, filtered
and associated graded algebras, Ore extensions. We will also
consider briefly the isoperimetric profile of modules.
    \item In the third section we compute the isoperimetric
    profile of many algebras, providing a proof of Theorem \ref{theorem5}.
    \item In the fourth section we discuss the relation of the
    isoperimetric profile with other invariants for algebras. In
    particular we study its relation with the lower
    transcendence degree introduced by Zhang in \cite{zhang},
    and we derive from this some consequences on amenability of
    algebras. Also, we study its relation with the growth, answering a question in \cite{gromov} Section 1.9.
\end{itemize}

\section{Definitions and basic properties}

In this section we give basic definitions and properties.

\subsection{The Isoperimetric Profile}

Unless otherwise stated, by an \textit{algebra} $A$ we will mean an
infinite dimensional associative algebra with unit $1$ over a fixed
field $K$ of characteristic $0$.\newline

Given two subspaces $V$ and $W$ of an algebra $A$ we will denote
the quotient space $V/(V\cap W)$ simply by $V/W$. Also, given a
subset $S$ of $A$ and a subspace $V$ of $A$ we define $SV:=span_K
\{sv|s\in S, v\in V\}$.

In this notation, given a subspace $V$ of $A$ and a subset $S$ of
$A$, the \textit{boundary} of $V$ with respect to $S$ is defined
by
$$
\partial_S(V):=SV/V.
$$

We will denote the dimension over $K$ of a subspace $V$ of $A$ by
$|V|$. Also, for any finite set $S$ we denote by $|S|$ its
cardinality. Hopefully this will not cause any confusion.

We are interested in the dimension of the boundary, hence we can
always assume that $1$ (the identity of $A$) is in $S$, since
$$
\partial_{S\cup\{1\}}(V)=(S\cup\{1\})V/V=(SV+V)/V\cong SV/(SV\cap V)=SV/V=\partial_S(V).
$$

It's easy to show the following inequality:
\begin{equation} \label{basic}
\tag{$\bullet$} |\partial_{ST}(V)|\leq
|\partial_{S}(V)|+|S||\partial_{T}(V)|,
\end{equation}
where $S$ and $T$ are finite subsets of $A$. Notice also that if
$S$ is a finite subset of $A$ and $V=span_KS=KS$, then
$\partial_S(W)=\partial_V(W)$ for all subspaces $W$ of $A$. Hence
the same inequality is true if we assume $S$ and $T$ to be finite
dimensional subspaces.

\begin{Def}
We define a \textit{subframe} of an algebra $A$ to be a finite
dimensional subspace containing the identity and a \textit{frame} to
be a subframe which generates the algebra. (see \cite{zhang})
\end{Def}
\begin{Remrk}
The previous discussion shows that as long as we are interested in
the dimension of the boundary $\partial_V(W)$, instead of taking
an arbitrary finite dimensional subspace $V$ of an algebra $A$, we
can take a subframe, without loosing anything.
\end{Remrk}

\begin{Conve}
In the rest of the paper by a \textit{subspace} we will always mean
\textit{finite dimensional subspace}, unless otherwise specified.
\end{Conve}

Given a subframe $V$ of $A$, in the Introduction we defined the
\textit{isoperimetric profile of $A$ with respect to $V$} (see
\cite{gromov}) to be the maximal function $I_{*}$ such that all
finite dimensional subspaces $W\subset A$ satisfy the
\textit{isoperimetric inequality}
$$
I_{*}(|W|;A,V)=I_{*}(|W|)\leq |\partial_V(W)|.
$$
Notice that for any $n\in \mathbb{N}$
$$
I_{*}(n;A,V)=I_{*}(n)=\inf |\partial_V(W)|,
$$
where the infimum is taken over all subspaces $W$ of $A$ of
dimension $n$.

We are interested in the asymptotic behavior of the function
$I_*$.
\begin{Def}
Given two functions $f_1,f_2 : \mathbb{R}_{+}\rightarrow
\mathbb{R}_+$ we say that $f_1$ is \textit{asymptotically faster}
then $f_2$, and we write $f_1\succeq f_2$, if there exist positive
constants $C_1$ and $C_2$ such that $f_1(C_1 x)\geq C_2 f_2(x)$ for
all $x\in \mathbb{R}_{+}$. We say $f_1$ is \textit{asymptotically
equivalent} to $f_2$, and we write $f_1 \sim f_2$, if $f_1\succeq
f_2$ and $f_2\succeq f_1$.
\end{Def}
\begin{Remrk}
We can always consider the function $I_*(\,\,\cdot \,\,)$ as a
function on $\mathbb{R}_{+}$, simply defining for $r\in
\mathbb{R}_{+}$, $I_*(r):=I_*(\lfloor r \rfloor)$, where $\lfloor
r \rfloor$ denotes the maximal integer $\leq r$. We will often do
it, without mentioning it explicitly.
\end{Remrk}
\begin{Def}
We say that an algebra $A$ \textit{has an isoperimetric profile}
if there exists a subframe $V$ of $A$ such that for any other
subframe $W$ of $A$ we have
$$
I_*(n;A,W)\preceq I_*(n;A,V).
$$
Otherwise we say that $A$ has no isoperimetric profile.

In case $A$ has an isoperimetric profile, we will refer to this
function, or its asymptotic behavior, as \textit{the isoperimetric
profile of $A$}, and we'll denote it also by $I_*(A)$. If the
subframe $V$ of $A$ is such that $I_*(n;A,V)$ is the isoperimetric
profile of $A$ we will say that $V$ \textit{measures} the profile
of $A$.
\end{Def}

First of all we want to observe that an arbitrary finitely
generated algebra has an isoperimetric profile. The following
proposition follows easily from (\ref{basic}).
\begin{Prop} \label{geometricindependence}
If $V$ and $W$ are two frames of $A$, then $I_*(\,\,\cdot \,\,
;A,V)\sim I_*(\,\,\cdot \,\, ;A,W)$.
\end{Prop}

Observe that in a finitely generated algebra $A$, any subframe $V$
is contained in a frame $W$, and obviously $I_*(n;A,V)\leq
I_*(n;A,W)$. This together with the previous proposition shows
that $A$ has an isoperimetric profile, and any frame of $A$
measures $I_*(A)$.

We will see later examples of algebras with an isoperimetric
profile which are not finitely generated (see Example
\ref{examplenotaffine}), and we will give also an example of an
algebra which has no isoperimetric profile (see Example
\ref{examplenoprofile}).

\subsection{Isoperimetric profile and Amenability}

In a way, the isoperimetric profile measures the degree of
amenability of an algebra.

\begin{Def}
We say that an algebra $A$ is \textit{amenable} if for each
$\epsilon >0$ and any subframe $V$ of $A$, there exists a subframe
$W$ of $A$ with $|VW|\leq (1+\epsilon)|W|$. This is the so called
\textit{F\o lner condition}.

We will see a lot of examples of amenable algebras in the rest of
the paper.
\end{Def}
Notice that the F\o lner condition can be restated in the
following way using the boundary: for any subspace $V\subset A$
and $\epsilon
>0$ there exists a subspace $W\subset A$ such that
$|\partial_V(W)|/|W|\leq \epsilon$. The following proposition
follows easily from the definitions.

\begin{Prop} \label{amenable}
An algebra $A$ is amenable if and only if $I_*(n;A,V)\precneqq n$
for any subframe $V$ of $A$.
\end{Prop}

The following corollaries are immediate.
\begin{Cor} \label{nonamenable}
An algebra $A$ is nonamenable if and only if $A$ has isoperimetric
profile $I_*(n;A)\sim n$.
\end{Cor}
\begin{Cor}
If all the finitely generated subalgebras of an algebra $A$ are
amenable, then $A$ is amenable.
\end{Cor}

\begin{Rem} \label{remdirectsum}
The converse of the previous corollary is not true. For example, we
will show later in the paper that the algebra $A=K[x,y]\oplus
K\langle w,z \rangle$ is amenable, since we'll prove (see
Proposition \ref{directsum} and Proposition \ref{polynomialalg})
that $I_*(A)\preceq I_*(K[x,y])\sim n^{1/2}$. But it's known (cf.
\cite{tullio}) that the finitely generated subalgebra $K\langle w,z
\rangle$ (a free algebra of rank 2) is not amenable.
\end{Rem}

\subsection{Orderable semigroups and the algebra of polynomials}

Let $\Gamma$ be an infinite semigroup generated by a finite subset
$S$. Let $B(n):=\cup_{i=0}^{n}S^i$, where $S^{0}=\{1\}$ and $1$ is
the identity element of $\Gamma$. Define $\Phi(\lambda):=\min
\{n\in \mathbb{N}\mid |B(n)|>\lambda\}$ for $\lambda>0$. This is
the inverse function of the growth of $\Gamma$.

The following result is due to Coulhon and Saloff-Coste. Here they
use a slightly different definition of the boundary:
$$
\delta_S(\Omega):=\{\gamma\in \Omega\mid \text{ there exists }s\in
S\text { such that } s\gamma\notin \Omega\},
$$
where $\Omega$ is a finite subset of $\Gamma$.
\begin{Thm}[Coulhon, Saloff-Coste] \label{coulsaloff} Let $\Gamma$ be an infinite
semigroup generated by a finite subset $S$. For any finite
non-empty subset $\Omega$ of $\Gamma$ we have
$$
|\partial_{S}(\Omega)|\geq \frac{|\delta_S(\Omega)|}{|S|}\geq
\frac{|\Omega|}{4|S|^2\Phi(2|\Omega|)}.
$$
\end{Thm}
The first inequality follows from the definitions of the boundaries.
For the second one, in \cite{pittetcoste} there is a short proof for
groups: this proof works verbatim for semigroups.

It's easy to see that the free abelian semigroup on $d\in
\mathbb{N}$ generators $\mathbb{Z}^{d}_{\geq 0}$ has isoperimetric
profile $I_{\circ}(n;\mathbb{Z}^{d}_{\geq 0})\sim
n^{\frac{d-1}{d}}$. The lower bound is given by Theorem
\ref{coulsaloff}. Considering the hypercubes, we easily get the
upper bound.

Now there is a theorem by Gromov (see \cite{gromov}, Section 3) that
states that the isoperimetric profile of an orderable semigroup is
asymptotically equivalent to the isoperimetric profile of its
semigroup algebra. These two together give the following fundamental
computation, that we state here as a proposition for our
convenience.
\begin{Prop}[\cite{gromov}] \label{polynomialalg}
The isoperimetric profile of the algebra of polynomials
$A=K[x_1,\dots,x_d]$ is $I_*(n;A)\sim n^{\frac{d-1}{d}}$.
\end{Prop}
We can now give an example of an algebra which has no
isoperimetric profile.
\begin{Ex} \label{examplenoprofile}
Consider the algebra $A=K[x_1,x_2,\dots]$ of polynomials in
infinitely many variables. For any $d\in \mathbb{N}$, call
$W_d=span_K\{x_1,\dots,x_d\}$. We can consider the vector space
$V_{n}^{(d)}= span_K\{x_{1}^{m_1}\cdots x_{d}^{m_d}\mid  m_i\leq n-1
\text{ for all $i$}\}$. We have $|V_{n}^{(d)}|=n^d$ and
$|\partial_{W_d}(V_{n}^{(d)})|=dn^{d-1}=
d|V_{n}^{(d)}|^{\frac{d-1}{d}}$, which easily implies the upper
bound
$$
I_*(n;A,W_d)\preceq n^{\frac{d-1}{d}}.
$$

Now $A$ is a free $K[x_1,\dots,x_d]$-module, hence we can apply
Proposition \ref{modules}, which we will prove later, to get
$$
n^{\frac{d-1}{d}}\sim I_*(n;K[x_1,\dots,x_d],W_d)\preceq
I_*(n;A,W_d),
$$
giving $I_*(n;A,W_d)\sim n^{\frac{d-1}{d}}$.

Notice that any subspace $W\subset A$ is contained in $W_{d}^{m}$
for some $d$ and $m\in \mathbb{N}$. Hence we can apply (\ref{basic})
to see that
$$
I_*(n;A,W)\preceq I_*(n;A,W_d)\sim n^{\frac{d-1}{d}}.
$$

This shows that $A$ cannot have an isoperimetric profile.
\end{Ex}

\section{Ring-theoretic constructions}

In this section we study the behavior of the isoperimetric profile
under various ring-theoretic constructions.

\subsection{Subalgebras and homomorphic images}

In general, the isoperimetric profile for algebras does not
decrease when passing to subalgebras or homomorphic images.
\begin{Lem}
If $A$ and $B$ are two algebras, $V$ is a subframe of $A$ and $W$
is a subframe of $B$, then $I_*(n;A\oplus B,V+W)\leq I_*(n;A,V)$
and $I_*(n;A\oplus B,V+W)\leq I_*(n;B,W)$.
\end{Lem}
\begin{proof}
We identify $A$ and $B$ with their obvious copies in $A\oplus B$.
Let $V$ be a subframe of $A$, $W$ a subframe of $B$ and let
$Z\subset A$ be any subspace. We have
$$
|\partial_{V+W}(Z)|=|\partial_{V}(Z)|,
$$
where the second boundary is in the algebra $A$. This proves the
first inequality. The second is proved in the same way.
\end{proof}

We have the following immediate consequence.
\begin{Prop} \label{directsum}
If $A$ and $B$ are two finitely generated algebras, then
$I_*(A\oplus B)\preceq I_*(A)$ and $I_*(A\oplus B)\preceq I_*(B)$.
\end{Prop}

Observe that $A$ is a subalgebra of $A\oplus B$, and also $A$ is
isomorphic to a homomorphic image of $A\oplus B$. If we now consider
a direct sum $A\oplus B$ of two finitely generated algebras with
$I_*(A)\precneqq I_*(B)$ (cf. Remark \ref{remdirectsum}), it follows
immediately from the previous proposition that we do not have in
general inequality for subalgebras and homomorphic images.

From this and what we saw in the previous sections it follows for
example that amenability for algebras does not pass to quotients
and subalgebras (see also \cite{tullio}).

We already observed in the Introduction (after Theorem
\ref{theorem3}) that this phenomenon can occur also when we deal
with domains.

\subsection{Localization}

The isoperimetric profile behaves well with nice localizations.

If $A$ is an algebra, a \textit{right Ore set} $\Omega\subseteq A$
is a multiplicative closed subset of $A$ which satisfies the
\textit{right Ore condition}, i.e. $cA\cap a\Omega \neq \emptyset$
for all $c\in \Omega$ and $a\in A$. If all the elements of
$\Omega$ are regular, we can consider the ring of right fractions
$A\Omega^{-1}$, and identify $A$ with the subset $\{a/1\mid a\in
A\}\subseteq A\Omega^{-1}$.

There are analogous left versions of these notions.

Notice that we will have slightly different results for the left and
the right cases in this section. This depends on the fact that the
definition of the boundary is not symmetric.
\begin{Lem} \label{lemmalocal}
Let $A$ be an algebra and let $\Omega$ be a right Ore set of
regular elements in (i) and (ii) and a left Ore set of regular
elements in (iii).
\begin{itemize}
  \item[(i)] If $V$ is a subframe of $A$, then
    $$
    I_*(n;A,V)=I_*(n;A\Omega^{-1},V).
    $$
  \item[(ii)] If $W$ is a subframe of $A\Omega^{-1}$, then we
can find an $m\in \Omega$ such that $Wm\subset A\subset
A\Omega^{-1}$. For any such $m$
    $$
    I_*(n;A\Omega^{-1},W)\leq I_*(n;A,Wm+K).
    $$
  \item[(iii)]  If $W$ is a subframe of $\Omega^{-1}A$, we can
find an $m\in \Omega$ such that $mW\subset A\subset \Omega^{-1}A$.
For any such $m$
    $$
    I_*(n;\Omega^{-1}A,W)\leq I_*(n;A,mW+K).
    $$
\end{itemize}
\end{Lem}
\begin{proof}
(i) Let $V$ be a subframe of $A$. Of course $V$ is also a subframe
of $A\Omega^{-1}$. Given any subspace $Z$ of $A\Omega^{-1}$,
clearly we can find an element $m\in \Omega$ such that
$Zm\subseteq A\subseteq A\Omega^{-1}$. We have
$$
|\partial_V(Zm)|=|VZm|-|Zm| = |VZ|-|Z|= |\partial_{V}(Z)|.
$$
Hence
$$
I_*(n;A,V)\leq I_*(n;A\Omega^{-1},V),
$$
which implies
$$
I_*(n;A,V)=I_*(n;A\Omega^{-1},V).
$$

(ii) Given now a subframe $W$ of $A\Omega^{-1}$, again we can find
an $m\in \Omega$ such that $Wm\subset A\subset A\Omega^{-1}$. If
$Z$ is a subspace of $A$, we have
$$
|\partial_W(mZ)|=|WmZ|-|mZ|\leq|WmZ+Z|-|Z|=|\partial_{Wm+K}(Z)|.
$$

The above inequality shows that
$$
I_*(n;A\Omega^{-1},W)\leq I_*(n;A,Wm+K).
$$

(iii) Suppose that $W$ is a subframe of $\Omega^{-1}A$. As before
we can find an $m\in \Omega$ such that $mW\subset A\subset
\Omega^{-1}A$. If $Z$ is a subspace of $A$, we have
$$
|\partial_W(Z)|=|WZ|-|Z| \leq |mWZ+Z|-|Z|= |\partial_{mW+K}(Z)|.
$$

The above inequality gives
$$
I_*(n;\Omega^{-1}A,W)\leq I_*(n;A,mW+K).
$$
\end{proof}

The following corollary follows easily from this lemma. It's a more
general version of Theorem \ref{theorem1}.

\begin{Cor} \label{localization}
Let $A$ be an algebra and let $\Omega$ be a right Ore set of
regular elements in (i) and a left Ore set of regular elements in
(ii). Then
\begin{itemize}
    \item[(i)] $A$ has an isoperimetric profile if and only if $A\Omega^{-1}$
    does, and in this case $I_*(A)\sim I_*(A\Omega^{-1})$. Moreover, any
    subframe of $A$ that measures $I_*(A)$, measures also $I_*(A\Omega^{-1})$,
    and viceversa if $W$ measures $I_*(A\Omega^{-1})$, then for any $m\in \Omega$
    such that $Wm\subset A$, $Wm+K$ measures $I_*(A)$.
    \item[(ii)] If both $A$ and $\Omega^{-1}A$ have isoperimetric profiles, then
    $I_*(\Omega^{-1}A)\preceq I_*(A)$.
\end{itemize}
\end{Cor}

\begin{Remrk}
In \cite{zhang}, the remark after Proposition 2.1 may suggest that
$I_*(A)\preceq I_*(\Omega^{-1}A)$ is not true in general.
\end{Remrk}
We can now give an example of an algebra with an isoperimetric
profile, which is not finitely generated.
\begin{Ex} \label{examplenotaffine}
If $A=K[x_1,\dots ,x_d]$ is the algebra of polynomials in $d$
variables, then we already saw that $I_*(A)\sim n^{\frac{d-1}{d}}$.
If we denote as usual by $K(x_1,\dots,x_d)$ the quotient field of
$A$, using the previous corollary we have
$$
I_*(K(x_1,\dots,x_d))\sim n^{\frac{d-1}{d}}.
$$

Notice that $K(x_1,\dots,x_d)$ is not finitely generated as an
algebra.
\end{Ex}

Another immediate consequence of this corollary is for example that
$I_*(K[x_{1}^{\pm 1},\dots, x_{d}^{\pm 1}])\sim n^{\frac{d-1}{d}}$,
where $K[x_{1}^{\pm 1},\dots, x_{d}^{\pm 1}]$ is the algebra of
Laurent polynomials in $d$ variables (see \cite{gromov}).

The following consequences on the amenability of a localization
follow easily from Lemma \ref{lemmalocal} and Proposition
\ref{amenable}.

\begin{Cor}
Let $A$ be an algebra and let $\Omega$ be a right Ore set of
regular elements in (i) and a left Ore set of regular elements in
(ii). Then
\begin{itemize}
  \item[(i)] $A$ is amenable if and only if $A\Omega^{-1}$ is
  amenable.
  \item[(ii)] If $A$ is amenable, then $\Omega^{-1}A$ is amenable.
\end{itemize}
\end{Cor}

\subsection{Subadditivity}

\begin{Def}
We say that a function $f:\mathbb{R}_{+}\rightarrow
\mathbb{R}_{+}$ is (\textit{asymptotically}) \textit{subadditive}
if there exist positive constants $C_1,C_2>0$ such that for every
finite set of positive real numbers $r_1,\dots,r_k$ we have
$$
C_2f(C_1(r_1+\cdots +r_k))\leq f(r_1)+\cdots +f(r_2).
$$
\end{Def}
\begin{Ex}
The function $f(x)=x^{\alpha}$ for $0\leq \alpha \leq 1$ is
subadditive with constants $C_1=C_2=1$.

For example the isoperimetric profile of an infinite group is
subadditive with constants $C_1=C_2=1$ (cf. \cite{gromov}).
\end{Ex}

The following lemma motivates our definition of subadditivity.
\begin{Lem} \label{subadditive}
Given two functions $f,g: \mathbb{R}_{+} \rightarrow
\mathbb{R}_{+}$, if $f\sim g$, then $f$ is subadditive if and only
if $g$ is.
\end{Lem}

We now show that the isoperimetric profile of a domain is
subadditive. We need the following proposition, which was showed
to me by Zelmanov.
\begin{Prop}[Zelmanov] \label{zelmanov}
Let $A$ be a domain over $K$, and let $V$ and $W$ be finite
dimensional subspaces of $A$, with $|V|=m$ and $|W|=n$. If $V\cap
Wa\neq \{0\}$ for all $a\in A\setminus \{0\}$, then $A$ is
algebraic of bounded degree.
\end{Prop}
To prove this proposition we need the following lemma.
\begin{Lem}
In the hypothesis of the previous proposition, let $\{w_1,\dots
,w_n\}$ be a basis of $W$. Then for any nonzero element $a\in A$
there exist polynomials $f_1(t),\dots ,f_n(t)$, not all zero and
all of degree $\leq m$ such that
$$
w_1f_1(a)+\cdots +w_nf_n(a)=0.
$$
\end{Lem}
\begin{proof}
Given $0\neq a\in A$, we have $V\cap W1\neq \{0\},\, V\cap Wa\neq
\{0\},\dots, \, V\cap Wa^m\neq \{0\}.$ Hence there are
coefficients $\alpha_{ij}\in K$ such that
\begin{eqnarray*}
0 & \neq & \alpha_{01}w_1+\cdots + \alpha_{0n}w_n\in V,\\
0 & \neq & \alpha_{11}w_1a+\cdots + \alpha_{1n}w_na\in V,\\
 & & \vdots\\
0 & \neq & \alpha_{m1}w_1a^m+\cdots + \alpha_{mn}w_na^m\in V.\\
\end{eqnarray*}

Since $|V|=m$, these elements are linearly dependent, hence there
exist $\beta_0,\dots,\beta_m$ not all zero such that
$$
\beta_0(\alpha_{01}w_1+\cdots +\alpha_{0n}w_n)+
\beta_1(\alpha_{11}w_1a+\cdots +\alpha_{1n}w_na)+
\cdots\qquad\qquad\qquad
$$
$$
\qquad\qquad\qquad\qquad\qquad\cdots
+\beta_m(\alpha_{m1}w_1a^m+\cdots +\alpha_{mn}w_na^m)=0,
$$
which implies
$$
w_1(\beta_0\alpha_{01}+\beta_1\alpha_{11}a+\cdots
+\beta_m\alpha_{m1}a^m)
+w_2(\beta_0\alpha_{02}+\beta_1\alpha_{12}a+\cdots
+\beta_m\alpha_{m2}a^m)+\cdots\qquad\qquad\qquad
$$
$$
\qquad\qquad\qquad\qquad\qquad\qquad\cdots
+w_n(\beta_0\alpha_{0n}+\beta_1\alpha_{1n}a+\cdots
+\beta_m\alpha_{mn}a^m)=0.
$$

We set $f_i(t):=\beta_0\alpha_{0i}+\beta_1\alpha_{1i}t+\cdots
+\beta_m\alpha_{mi}t^m$ for $i=1,\dots,n$. If all the $f_i$'s are
zero, then $\beta_i\alpha_{ij}=0$ for $0\leq i\leq m$ and $1\leq
j\leq n$. But each row $(\alpha_{i0},\dots,\alpha_{in})$ is not
the zero vector, because $\sum_j \alpha_{ij}w_ja^i\neq 0$. Hence
$\beta_i=0$ for all $i$, a contradiction.
\end{proof}
We can now prove the proposition.
\begin{proof}
Let $\{w_1,\dots ,w_n\}$ be a basis of $W$. By the lemma, for
$0\leq i\leq m$ we can find polynomials $f_{i1},\dots,f_{in}$, not
all zero and of degree $\leq m$ such that
\begin{equation} \label{eqzelmanov}
\tag{*} \sum_jw_jf_{ij}\left(a^{(m+1)^i}\right)=0.
\end{equation}

We have
$$
\det \left\| f_{ij}\left(a^{(m+1)^{i}}\right) \right\|=0.
$$

We got in this way a polynomial of degree bounded by a function of
$m$ and $n$ only, satisfied by $a$. If this is not the zero
polynomial, we are done.

Suppose this is not the case. Let
$f_{ij}(t):=\alpha_{ij0}+\alpha_{ij1}t+\cdots +\alpha_{ijm}t^m$,
and suppose that
$$
\det \left\| f_{ij}\left(t^{(m+1)^{i}}\right) \right\|=0.
$$

Observe that in each row of the matrix $\left\| f_{ij}\left(t^{(m+1)^{i}}\right) \right\|$ there are at least two nonzero polynomials. In fact
we know that they are not all zero. If only one of them is zero, then the equation (\ref{eqzelmanov}) gives a zero divisor, which doesn't exist
by our assumption. Moreover, we can assume that in each row the entries have no common divisors of the form $t^k$ with $k\geq 1$, since
otherwise we can factor it out, preserving the relation (\ref{eqzelmanov}). Hence in particular in each row there is at least one polynomial
with nonzero constant term.

Since these rows are linearly dependent, we can take a minimal linearly dependent set of rows, call $r$ the cardinality of this set and call the
indices of these rows $j_1,j_2,\dots ,j_r$. By construction all the minors of order $r$ in these rows are zero. Considering these minors modulo
$t^{(m+1)^{j_1+1}}$ we can replace the coefficients in the first of our rows by their constant terms, still having the first row non zero and
depending on the others. Hence we can find polynomials $b(t),c_2(t),\dots,c_{r}(t)$ such that
$$
b(t)\alpha_{j_1k0}=\sum_{i=2}^{r}c_i(t)f_{j_ik}\left(t^{(m+1)^{j_i}}\right)
$$
for all $k=1,\dots,n$. By assumption $b(t)\neq 0$. Observe now
that (\ref{eqzelmanov}) implies
\begin{eqnarray*}
b(a)\left( \sum_{k=1}^nw_k\alpha_{j_1k0}\right) & = & \sum_{k=1}^nw_kb(a)\alpha_{j_1k0}\\
 & = & \sum_{k=1}^nw_k \sum_{i=2}^{r}c_i(a)f_{j_ik}\left(a^{(m+1)^{j_i}}\right)\\
 & = &  \sum_{i=2}^{r}c_i(a)\left( \sum_{k=1}^nw_k f_{j_ik}\left(a^{(m+1)^{j_i}}\right)
 \right)=0.
\end{eqnarray*}
Since $\sum_{k=1}^nw_k\alpha_{j_1k0}\neq 0$, we must have
$b(a)=0$. It's now clear that $b(t)$ also has degree bounded by a
function of $m$ and $n$ only. This completes the proof.
\end{proof}

The following lemma is crucial.
\begin{Lem}
If $A$ is an (infinite dimensional) division algebra, then given two
finite dimensional subspaces $V$ and $W\subset A$ there exists a
nonzero element $a\in A$ such that $V\cap Wa=\{0\}$.
\end{Lem}
\begin{proof}
Suppose the contrary. Then by the previous proposition we know
that $A$ is algebraic of bounded degree. Hence by a theorem of
Jacobson (see \cite{jacobson}) $A$ is locally finite, i.e. any
finitely generated subalgebra of $A$ is finite dimensional. But
for any nonzero $a\in A$ we have $v=wa$ for some nonzero $v\in V$
and some nonzero $w\in W$, i.e. $a=w^{-1}v$. Hence $a$ is
contained in the subalgebra generated by $V$ and $W$, which is
finite dimensional. This gives a contradiction, since $A$ is not
finite dimensional.
\end{proof}

We are now able to prove the main result of this subsection.
\begin{Thm} \label{domainsubadditive}
If $A$ is a nonamenable domain, then $I_*(A)$ is subadditive. If
$A$ is an amenable domain, then $I_*(A,V)$ is subadditive for any
subframe $V$ of $A$.
\end{Thm}
\begin{proof}
If $A$ is nonamenable, then by Corollary \ref{nonamenable}
$I_*(n;A)\sim n$, hence by Lemma \ref{subadditive} $I_*(A)$ is
subadditive.

If $A$ is amenable, then by Proposition \ref{amenable} we know
that $I_*(A,V)\precneqq n$ for any subframe $V$ of $A$. In this
case, we know that $A$ is a right Ore domain, hence it admits a
ring of quotients $D$, which is of course a division algebra. By
Lemma \ref{lemmalocal}, $I_*(n;A,V)=I_*(n;D,V)$, hence again by
Lemma \ref{subadditive} we reduced the problem to show that $D$
has a subadditive isoperimetric profile.

Let $r,s\in \mathbb{N}$, and consider two subspaces $W,Z\subset D$
with $|W|=r$ and $|Z|=s$. By the previous lemma, we can find an
element $a\in D$ such that $W\cap Za=\{0\}$. If now $V$ is any
subframe of $D$, we have
\begin{eqnarray*}
|\partial_V(W\oplus Za)| & = & |V(W\oplus Za)|-|W\oplus Za|\leq
|VW|+|VZa|-|W|-|Za|\\
 & = & |VW|+|VZ|-|W|-|Z|=
|\partial_V(W)|+|\partial_V(Z)|,
\end{eqnarray*}
which gives the subadditivity of $I_*(n;D,V)$.
\end{proof}
\begin{Ques}
Is the isoperimetric profile with respect to some subframe of an
algebra always subadditive?
\end{Ques}

\subsection{Free left modules over subalgebras}

We now study algebras which are a free left module over some subalgebra.

The proof of the following proposition is a modification of the
proof of Theorem 2.4 in \cite{zhang}.
\begin{Prop} \label{freeleftmod}
Suppose that $B\subset A$ is a subalgebra and $A$ is a free left
$B$-module. If $V$ is a subframe of $B$ and $I_*(B,V)$ is
subadditive, then $I_*(B,V)\preceq I_*(A,V)$.
\end{Prop}
\begin{proof}
We have $A=\bigoplus_i Ba_i$ where $a_i\in A$. Given any subspace
$W$ of $A$ we can find $a_1,\dots,a_n$ such that $W\subset
\bigoplus_{i=1}^{n}Ba_i$. We can choose a basis of $W$ of the form
$$
\{w_{i}^{1}a_1+y_{i}^{1}\}_{i=1}^{p_1}\cup
\{w_{i}^{2}a_2+y_{i}^{2}\}_{i=1}^{p_2}\cup\cdots \cup
\{w_{i}^{n}a_n+y_{i}^{n}\}_{i=1}^{p_n}
$$
where $w_{i}^{j}\in B$ and $y_{i}^{j}\in \bigoplus_{k>j}Ba_k$,
such that for each $j$, $\{w_{i}^{j}\}_{i=1}^{p_j}$ are linearly
independent. Notice that $\{w_{i}^{j}a_j+y_{i}^{j}\}_{i=1}^{p_j}$
corresponds to a basis of $(W\cup\bigoplus_{k\geq j}Ba_k )/(W\cup
\bigoplus_{k>j}Ba_k)$. Let $W_j'$ denote the subspace generated by
$\{w_{i}^{j}\}_{i=1}^{p_j}$ and let $W_j$ denote the subspace
generated by $\{w_{i}^{j}a_j+y_{i}^{j}\}_{i=1}^{p_j}$. Then
$$
W=W_1\oplus W_2\oplus \cdots \oplus W_n
$$

and hence
$$
|W|=\sum_j|W_j| = \sum_j|W_j'|.
$$
Let $V$ be a subframe of $B$. We have
$$
VW_1=\left\{xa_1 + y\mid x\in VW_1'\text{ and } y\in
\bigoplus_{i=2}^{n}Ba_i \right\}.
$$
Since
$$
\sum_{i=2}^{n}VW_i\subset
\bigoplus_{i=2}^{n}Ba_i\quad\text{and}\quad
\left(\bigoplus_{i=2}^{n}Ba_i\right)\cap Ba_1 =0,
$$
we have
$$
\left|\sum_{i=1}^{n}VW_i \right|\geq
|VW_1'|+\left|\sum_{i=2}^{n}VW_i \right|.
$$
By induction on $n$ we have
$$
\left|\sum_{i=1}^{n}VW_i \right|\geq \sum_{i=1}^{n}|VW_i'|.
$$

Using the hypothesis, this implies
\begin{eqnarray*}
|\partial_V(W)| & = & |VW|-|W|=
\left|\sum_{i=1}^{n}VW_i \right|-\sum_{i=1}^{n}|W_i|\\
 & \geq &
 \sum_{i=1}^{n}|VW_i'|-\sum_{i=1}^{n}|W_i'|=\sum_{i=1}^{n}|\partial_V(W_i')|\\
 & \geq & \sum_{i=1}^{n}I_*(|W_i'|;B,V)\geq
 C_2I_*(C_1\sum_{i=1}^{n}|W_i'|;B,V)=C_2I_*(C_1|W|;B,V),
\end{eqnarray*}
where $C_1$ and $C_2$ are two positive constants. Therefore
$$
I_*(B,V)\preceq I_*(A,V).
$$
\end{proof}

The following corollaries are immediate consequences of the
proposition.
\begin{Cor}
Suppose that $B\subset A$ is a subalgebra and $A$ is a free left
$B$-module. If both $A$ and $B$ have isoperimetric profiles, and
$I_*(B)$ is subadditive, then $I_*(B)\preceq I_*(A)$.
\end{Cor}
\begin{Cor}
Suppose that $B\subset A$ is a subalgebra and $A$ is a free left
$B$-module. If $A$ is amenable and $I_*(B)$ is subadditive, then $B$
is amenable.
\end{Cor}
Let's derive another easy consequence from the previous proposition,
which generalizes a result in \cite{elek2}.
\begin{Prop} \label{subdivisionalg}
If $B$ is a nonamenable division subalgebra of $A$, then $A$ is
nonamenable. If $B$ is an amenable division subalgebra of $A$,
then $I_*(B,V)\preceq I_*(A,V)$ for any subframe $V$ of $B$. In
particular, if both $A$ and $B$ have isoperimetric profiles, then
$I_*(B)\preceq I_*(A)$.
\end{Prop}
\begin{proof}
If $B$ is a nonamenable division subalgebra, then $A$ is a free
left $B$-module. By Theorem \ref{domainsubadditive}, $I_*(B,V)$ is
subadditive for any subframe $V$ that measures $I_*(n;B)\sim n$,
hence by Proposition \ref{freeleftmod}
$$
n\sim I_*(n;B,V)\preceq I_*(n;A,V)
$$
for any subframe $V$ of $B$ that measures $I_*(B)$. Hence
$I_*(n;A,V)\sim n$, and so $A$ is nonamenable by Corollary
\ref{nonamenable}.

If $B$ is an amenable division subalgebra, $A$ is again a free
left $B$-module. By Theorem \ref{domainsubadditive}, $I_*(B,V)$ is
subadditive for any subframe $V$, hence the result follows again
from Proposition \ref{freeleftmod}.
\end{proof}

We are now able to prove the following

\begin{Thm} \label{subOredomains}
Let $B\subset A$ be domains. If both $B$ and $A$ are right Ore,
then $I_*(B,V)\preceq I_*(A,V)$ for all subframes $V$ of $B$.
\end{Thm}
\begin{proof}
If we call $S$ and $D$ the right quotient division algebras of $B$
and $A$ respectively, by Lemma \ref{lemmalocal}, if $V$ is a
subframe of $B$ we have $I_*(n;B,V)= I_*(n;S,V)$ and $I_*(n;A,V)=
I_*(n;D,V)$. Since $I_*(S,V)$ is also subadditive, we can apply
Proposition \ref{subdivisionalg} to $S\subset D$ to get
$I_*(S,V)\preceq I_*(D,V)$. Now again by Lemma \ref{lemmalocal},
$I_*(B,V)\preceq I_*(A,V)$.
\end{proof}

The following corollary is a more general form of Theorem
\ref{theorem3}.

\begin{Cor}
If $B\subset A$ are domains, $B$ is right Ore and both $A$ and $B$
have isoperimetric profiles, then $I_*(B)\preceq I_*(A)$.
\end{Cor}
\begin{proof}
If $A$ is nonamenable, $I_*(n;A)\sim n$ and there is nothing to
prove. Otherwise, the result follows from the previous theorem.
\end{proof}
\begin{Rem}
Notice that the hypothesis on $B$ of being right Ore cannot be
dropped. For example we already observed in the Introduction that
the quotient division algebra of the Weyl algebra $A_1$ is amenable,
but it contains a subalgebra isomorphic to a free algebra in two
variables (see \cite{makar}). We show now that this is the only case
that can occur.
\end{Rem}

By a theorem of Jategaonkar (\cite{jategaonkar}), a domain which is
not Ore must contain a subalgebra isomorphic to a noncommutative
free algebra. This and the previous proposition imply the following
corollary which is a more general form of Theorem \ref{theorem4}.
\begin{Cor}
If $A$ is an amenable domain, then for any subdomain $B$ of $A$ we
have $I_*(B,V)\preceq I_*(A,V)$ for all subframes $V$ of $B$ if
and only if $A$ does not contain a subalgebra isomorphic to a
noncommutative free algebra.
\end{Cor}

\subsection{Finite modules over subalgebras}

Suppose that $B$ is a subalgebra of an algebra $A$. Assume that
$A$ is a finite right $B$-module, i.e. $A=WB$, where $W$ is a
subframe of $A$. We want to compare the isoperimetric profiles of
$A$ and $B$.

The following proposition generalizes some of the results in
\cite{elek2}.

\begin{Prop} \label{bigcorollary}
Let $A$ be an algebra.
\begin{itemize}
    \item[(1)] Let $B$ be a subalgebra of $A$ such that $A$ is a
    finite free right $B$-module. If $B$ is amenable, then $A$ is also
    amenable. If both $A$ and $B$ have isoperimetric profiles, then
    $I_*(A)\preceq I_*(B)$. If moreover $B$ has a subadditive isoperimetric
    profile and $A$ is also a free left $B$-module, then $I_*(A)\sim
    I_*(B)$.
    \item[(2)] Let $B$ be a division subalgebra of $A$ and let $A$ be a
    finite right $B$-module. If $B$ is amenable, then $A$ is also amenable.
    If both $A$ and $B$ have isoperimetric profiles,
    then $I_*(A)\sim I_*(B)$.
    \item[(3)] Let $B$ be a finite dimensional algebra and $A$ an
    algebra. If $A$ is amenable, then $A\otimes B$ is also amenable. If both $A$ and $A\otimes B$
    have isoperimetric profiles, then $I_*(A\otimes B)\preceq I_*(A)$. If moreover $A$ has
    a subadditive isoperimetric profile, then $I_*(A)\sim I_*(A\otimes B)$.
    \item[(4)] Let $M_n(A)$ be the algebra of $n\times n$ matrices
    over $A$. If $A$ is amenable, then $M_n(A)$ is also amenable.
    If both $A$ and $M_n(A)$ have isoperimetric profiles, then $I_*(M_n(A))\preceq I_*(A)$.
    If moreover $A$ has a subadditive isoperimetric profile, then $I_*(A)\sim I_*(M_n(A))$.
    \item[(5)] Let $G$ be a finite group and $A\ast G$ a skew
    group ring. If $A$ is amenable, then also $A\ast G$ is amenable.
    If both $A$ and $A\ast G$ have isoperimetric profiles, then $I_*(A\ast G)\preceq
    I_*(A)$. If moreover $A$ has subadditive isoperimetric profile, then $I_*(A)\sim I_*(A\ast G)$.
\end{itemize}
\end{Prop}

First we need a lemma.

\begin{Lem}
Let $V,W,Z\subset A$ be subspaces of $A$. Then
$$
\left| \frac{ZVW}{ZW} \right|\leq |Z| \left| \frac{VW}{W} \right|
$$
\end{Lem}
\begin{proof}
Let $v_1,\dots,v_m$ be a basis of $V$ and $w_1,\dots,w_n$ a basis
of $W$. Then the products $v_iw_j$ span $VW$. Clearly at most
$|VW/W|=|\partial_V(W)|$ of these products are not in $W$. For
each of them, multiplying on the left by elements of $Z$, we get
at most $|Z|$ products which do not fall into $ZW$. This proves
the result.
\end{proof}
The previous proposition follows from the following lemma together
with Propositions \ref{freeleftmod}, \ref{finiterightmod} and
\ref{subadditive}.

\begin{Lem} \label{finiterightmod}
Let $B$ be a subalgebra of an algebra $A$, let $V$ be a subframe
of $A$ and let $A$ be a finite right $B$-module. \begin{itemize}
    \item [(i)] If there exist positive constants $C_1$ and $C_2$
    such that $$C_1I_*(C_2n;A,V)\leq I_*(n+r;A,V)$$ for any $n,r\in
    \mathbb{N}$, then $I_*(A,V)\preceq I_*(B,V_1)$ for some subframe $V_1$ of $B$.
    \item [(ii)] If $A$ is also free as right $B$-module, then $I_*(A,V)\preceq I_*(B,V_1)$
    for some subframe $V_1$ of $B$.
\end{itemize}
\end{Lem}
\begin{proof}
Since $A$ is a finite right $B$-module, there exists a subframe
$W$ of $A$ such that $A=WB$. It's clear that given the subframe
$V$ of $A$ there exists a subframe $V_1$ of $B$ such that
$VW\subseteq WV_1$. For any subspace $Z$ of $B$, using the
previous lemma, we get
\begin{eqnarray*}
I_*(|WZ|;A,V) & \leq &
|\partial_{V}(WZ)|=\left|\frac{VWZ}{WZ}\right|
    \leq  \left|\frac{WV_1Z}{WZ}\right|\\
 & \leq & |W|\left|\frac{V_1Z}{Z}\right| =|W||\partial_{V_1}(Z)|.
\end{eqnarray*}
Now the hypothesis in (i) gives
$$
C_1I_*(C_2|Z|;A,V)\leq I_*(|WZ|;A,V)\leq |W||\partial_{V_1}(Z)|,
$$
which implies $I_*(A,V)\preceq I_*(B,V_1)$.

In (ii), if $A=\oplus_{i=1}^{k}w_iB$, we choose $W$ to be the span
of $\{1=w_1,w_2,\dots,w_k\}$. Then
$$
I_*(|W||Z|;A,V)=I_*(|WZ|;A,V)\leq |W||\partial_{V_1}(Z)|,
$$
which again gives $I_*(A,V)\preceq I_*(B,V_1)$.
\end{proof}
\begin{Remrk}
Notice that the hypothesis in (i) of this lemma is a generalization
of the property of being weakly monotone increasing. All the
isoperimetric profiles we know satisfy this property.
\end{Remrk}

\begin{Ques}
Is it true that $I_*(A)$ satisfies the property in (i) for any
algebra $A$? Is it true if $A$ is a domain?
\end{Ques}

We are now able to prove the following corollary (cf.
\cite{zhang}, Corollary 3.3).
\begin{Cor}
Let $B\subset A$ be prime right Goldie algebras with isoperimetric
profiles, and suppose that $I_*(B)$ is subadditive. Then
$I_*(B)\preceq I_*(A)$. If moreover $A$ is a finite right
$B$-module and $B$ is artinian, then $I_*(A)\sim I_*(B)$.
\end{Cor}
\begin{proof}
By Goldie's Theorem, $A$ has a right quotient ring which is a simple
artinian algebra. Hence by Corollary \ref{localization} we may
assume that $A$ is a simple artinian ring $M_n(A')$ for some
division algebra $A'$. By Proposition 3.1.16 in \cite{macconnell},
the quotient ring $Q$ of $B$ embeds into $M_k(A')$ for some $k\leq
n$. Therefore by Corollary \ref{localization} and Proposition
\ref{bigcorollary}, (4), we may assume that $B$ is a division
algebra. Whence the first statement follows from Proposition
\ref{subdivisionalg}.

If $B$ is artinian and $A$ is finite as $B$-module, then $A$ is
artinian. Therefore the second statement follows from Lemma
\ref{finiterightmod} and Proposition \ref{bigcorollary}, (4).
\end{proof}

\subsection{Tensor products}

In this section we study the behavior of the isoperimetric profile
with respect to tensor products.
\begin{Prop} \label{tensorprodprofile}
Let $A$ and $B$ be two $K$-algebras, and let $V_A$ and $V_B$ be two
subframes of $A$ and $B$ respectively. If $V:=V_A\otimes 1+ 1\otimes
V_B$, then
$$
I_*(nm;A\otimes_KB,V)\leq mI_*(n;A,V_A)+nI_*(m;B,V_B).
$$
\end{Prop}
\begin{proof}
Given any two subspaces $W\subset A$ and $Z\subset B$, we have
\begin{eqnarray*}
I_*(|W||Z|;A\otimes_KB,V) & \leq & |\partial_V(W\otimes Z)|  =
\left|\frac{V_AW\otimes Z+ W\otimes V_BZ}{W\otimes Z} \right|\\
 & \leq & \left|\frac{V_AW\otimes Z}{W\otimes
Z} \right|+\left|\frac{W\otimes V_BZ}{W\otimes Z} \right| \\
 & = & |Z||\partial_{V_A}(W)|+ |W||\partial_{V_B}(Z)|,
\end{eqnarray*}
which gives the result.
\end{proof}
\begin{Cor} \label{tensorprodpolyn}
Let $A$ and $B$ be two $K$-algebras, let $V_A$ and $V_B$ be two
subframes of $A$ and $B$ respectively, and let $V:=V_A\otimes 1+
1\otimes V_B$. If $I_*(n;A,V_A)\preceq n^{1-1/r}$ and
$I_*(n;B,V_B)\preceq n^{1-1/s}$ for some real numbers $s\geq r\geq
1$, then
$$
I_*(n;A\otimes_KB,V)\preceq n^{1-\frac{1}{r+s}}.
$$
\end{Cor}
\begin{proof}
Given $t\in \mathbb{R}$, $0<t<1$ the previous proposition implies
\begin{eqnarray*}
I_*(n;A\otimes_KB,V) & \preceq &
n^{t}I_*(n^{1-t};A,V_A)+n^{1-t}I_*(n^{t};B,V_B)\\
 & \preceq & n^{t+(1-t)(1-1/r)}+ n^{t+(1-t)(1-1/s)}.
\end{eqnarray*}

Substituting $t=r/(r+s)$ we get
$$
I_*(n;A\otimes_KB,V)\preceq
n^{\frac{r}{r+s}+\frac{s}{r+s}\frac{r-1}{r}} +
n^{\frac{r}{r+s}+\frac{s}{r+s}\frac{s-1}{s}}\preceq
n^{\frac{r+s-1}{r+s}},
$$
since $s\geq r$, hence both the exponents in the sum above are
less or equal then the exponent $(r+s-1)/(r+s)$.
\end{proof}

We have also the following immediate consequence of Proposition
\ref{freeleftmod}.
\begin{Prop} \label{tensorprodlower}
If $A$ and $B$ are two $K$-algebras, $V$ is a subframe of $A$ and
$I_*(A,V)$ is subadditive, then
$$
I_*(A,V)\preceq I_*(A\otimes_KB,V\otimes 1).
$$
\end{Prop}

The relation given in Proposition \ref{tensorprodprofile} looks a
bit strange. A more natural relation holds for F\o lner functions,
as we will see later in the paper.

\subsection{Filtered and Graded Algebras}

In this section we consider a filtration on $A$, i.e. a sequence
of subspaces $A_i$ of $A$
$$
A_0\subset A_1\subset A_2\subset \cdots \subset A,\quad
\bigcup_{n=0}^{\infty}A_n=A,
$$
with the property that $A_i A_j \subset A_{i+j}$ for all $i,j\geq
0$. We assume also that $A_0=K$ and that $A_1$ generates $A$.

Given a filtered algebra, we can consider its associated graded
algebra
$$
gr(A):=\bigoplus_{i\geq 0} A_i/A_{i-1},
$$
where we agree that $A_{-1}=\{0\}$. This is an algebra with the
multiplication derived by the rule
$$
[x+A_{i-1}]\cdot [y+A_{j-1}]= [xy+A_{i+j-1}].
$$

For any subframe $V\subset A_1$, we can view $V$ also as a subframe of $gr(A)$ via the identification $V\equiv (V\cap A_0)/A_{-1}\oplus
V/A_0=K\oplus V/K$.

The following theorem is a more general form of Theorem
\ref{theorem2}.
\begin{Thm} \label{gradedassociated}
If $A$ is an algebra with a filtration given as above, and $gr(A)$
is a domain, then $I_*(gr(A),V)\preceq I_*(A,V)$ for any subframe
$V\subset A_1$.
\end{Thm}
\begin{proof}
Given a subspace $W$ of $A$ we define $W_i=W\cap A_i$ and
$gr(W)=\bigoplus_{i\geq 0}W_i/W_{i-1}$. Observe that $gr(W)$ is a
finite dimensional subspace of $gr(A)$.

The first remark is that $|W|=|gr(W)|$: this can be seen looking
at a basis for $W_i$ and completing it to a basis of $W_{i+1}$ (if
$W_i\neq W_{i+1}$, otherwise look at the next index) for each $i$.
These basis elements clearly give a basis for $gr(W)$.

Now we want to compare $|\partial_{V}(W)|$ and
$|\partial_{V}(gr(W))|$. The remark we need is that for any finite
dimensional subspace $W$ of $A$, and any element $a\in A_1$ we
have
$$
|a\, gr(W)| = |gr(aW)|,
$$
where $a\, gr(W)$ is a short notation for $[a+A_0]gr(W)$.

We have
$$
a\, gr(W)= a\, \bigoplus_{i\geq 0}W_i/W_{i-1} = \bigoplus_{i\geq
0}\frac{aW_i}{aW_{i}\cap A_i},
$$
and
$$
gr(aW)=\bigoplus_{i\geq 0}\frac{aW\cap A_i}{aW\cap A_{i-1}},
$$
hence we want to show that
$$
\frac{aW\cap A_{i+1}}{aW\cap A_{i}}= \frac{aW_i}{aW_i\cap A_{i}}.
$$
Clearly $aW_i=a(W\cap A_i)\subseteq aW\cap A_{i+1}$. If the other
inclusion is false, then there exists $x\in W\setminus A_i$ with
$ax\in A_{i+1}$. So $x\in A_{i+p}\setminus A_i$ for some $p\geq
1$, with $ax\in A_{i+1}$. This gives a zero divisor in $gr(A)$,
which is a contradiction. Hence $aW_i=aW\cap A_{i+1}$.

Similarly $aW_i\cap A_i = aW\cap A_i$. In fact, it's obvious that
$aW_i\cap A_i \subseteq aW\cap A_i$. If the other inclusion is
false then there exists $x\in W\setminus A_i$ with $ax\in A_{i}$.
So $x\in A_{i+p}\setminus A_i$ for some $p\geq 1$, with $ax\in
A_{i}$. Again, this gives a zero divisor in $gr(A)$. Hence
$aW_i\cap A_i = aW\cap A_i$.

This proves the equality we wanted, giving $|gr(aW)|=|a\, gr(W)|$.

Let's now choose a basis $1=a_1,a_2,\dots ,a_r$ of $V$. We have
\begin{eqnarray*}
|\partial_{V}(gr(W))| & = & \left|\frac{V\,
gr(W)}{gr(W)}\right|=\left|\sum_j a_j \, gr(W)\right|-|gr(W)|=
\end{eqnarray*}
\begin{eqnarray*}
 & = & \left|\bigoplus_{i\geq 0}\sum_j\, \frac{a_j(W\cap
A_i)}{a_j(W\cap A_i)\cap A_i}\right|-|gr(W)|=\\
 & = & \sum_i \left|\sum_j\, \frac{a_j(W\cap A_i)}{a_j(W\cap A_i)\cap
A_i}\right|-|gr(W)|= \sum_i \left|\sum_j\, \frac{a_jW\cap
A_{i+1}}{a_jW\cap A_i}\right|-|gr(W)|=\\
& = & \sum_i \left|\sum_j\, \frac{a_jW\cap A_{i+1}}{(\sum_j
a_j)W\cap A_i}\right|-|gr(W)|\leq \sum_i \left| \frac{(\sum_j
a_j)W\cap A_{i+1}}{(\sum_j a_j)W\cap A_i}\right|-|gr(W)|=\\
 & = & \left|\bigoplus_i  \frac{(\sum_j a_j)W\cap A_{i+1}}{(\sum_j
a_j)W\cap A_i}\right|-|gr(W)|= |gr(VW)|-|gr(W)|=|VW|-|W|=\\
& = & |\partial_{V}(W)|.
\end{eqnarray*}

This gives $I_*(gr(A),V)\preceq I_*(A,V)$.
\end{proof}

In \cite{zhang} (see also \cite{zhang1}) Zhang considers a more
general setting.
\begin{Def}[\cite{zhang}]
Let $A$ and $B$ two $K$-algebras and let $\nu$ be a map from $A$ to
$B$. We call $\nu$ a \textit{valuation} from $A$ to $B$ if the
following conditions hold: \begin{itemize}
    \item[(v1)] $\nu(ta)=t\nu(a)$ for all $a\in A$ and $t\in K$;
    \item[(v2)] $\nu(a)\neq 0$ for all nonzero $a\in A$;
    \item[(v3)] for any $a,b\in A$, either $\nu(a)\nu(b)=\nu(ab)$
    or $\nu(a)\nu(b)=0$;
    \item[(v4)] for any subspace $W$ of $A$ $|\nu(W)|=|W|$.
\end{itemize}
\end{Def}

The main example of a valuation is the leading-term map of a
$\Gamma$-filtered algebra, where $\Gamma$ is any ordered semigroup.
Let $A$ be an algebra with a filtration $\{A_{\gamma}\mid \gamma\in
\Gamma\}$ of $A$, which satisfies the following conditions:
\begin{itemize}
    \item[(f0)] $K\subset A_{e}$ where $e$ is the unit of
    $\Gamma$;
    \item[(f1)] $A_{\alpha}\subset A_{\beta}$ for all $\alpha <
    \beta$ in $\Gamma$;
    \item[(f2)] $A_{\alpha}A_{\beta}\subset A_{\alpha\beta}$ for
    all $\alpha,\beta\in \Gamma$;
    \item[(f3)] $A=\cup_{\gamma\in
    \Gamma}(A_{\gamma}-A_{<\gamma})$, where $A_{<\gamma}=\cup_{\alpha <
    \gamma}A_{\alpha}$;
    \item[(f4)] $1\in A_{e}-A_{< e}$ (and hence $K\subset A_{e}-A_{<
    e}$).
\end{itemize}

Then we define the associated graded algebra to be
$gr(A):=\oplus_{\gamma\in \Gamma}A_{\gamma}/A_{< \gamma}$ with the
multiplication determined by $(a+A_{< \alpha})(b+A_{< \beta}) =ab+
A_{< \alpha\beta}$. Notice that this is the definition we gave
before with $\Gamma=\mathbb{N}$.

We define a map $\nu : A\rightarrow gr(A)$ by $\nu(a)=a+A_{<
\gamma}$ for all $a\in A_{\gamma}-A_{<\gamma}$. This $\nu$ is
called the \textit{leading-term map} of $A$ and it is easy to see
that it satisfies (v1,2,3,4) (see \cite{zhang1}, Section 6). If
also\begin{itemize}
    \item[(f5)] $gr(A)$ is a $\Gamma$-graded domain,
\end{itemize}
then $\nu(a)\nu(b)=0$ will not happen in (v3).

\begin{Thm}[compare to \cite{zhang}, Theorem 4.3] \label{valuationinequality}
If $A$ and $B$ are two $K$-algebras, and $\nu$ is a valuation from
$A$ to $B$, then
$$
I_*(B,\nu(V))\preceq I_*(A,V).
$$
\end{Thm}
\begin{proof}
If $W\subset A$, using Lemma 4.1, (3) in \cite{zhang}, we have
\begin{eqnarray*}
|\partial_{V}(W)| & = & |VW|-|W|=|\nu(VW)|-|\nu(W)|\\
 & \geq & |\nu(V)\nu(W)|-|\nu(Z)|=|\partial_{\nu(V)}(\nu(W))|,
\end{eqnarray*}
which gives $I_*(A,V)\succeq I_*(B,\nu(V))$, as we wanted.
\end{proof}

If $\Gamma$ is an ordered semigroup, $B$ is a $\Gamma$-filtered
graded $K$-algebra with the associated graded algebra $gr(B)$ and
$A$ is a $K$-algebra, then $A\otimes_K B$ is $\Gamma$-filtered, and
its associated graded is isomorphic to $A\otimes_K gr(B)$. Here is
another immediate consequence of Theorem \ref{valuationinequality}:
\begin{Cor}
If $\Gamma$ is an ordered semigroup, $A$ and $B$ are two finitely
generated $K$-algebras and $B$ is $\Gamma$-filtered, then
$$
I_*(A\otimes_K gr(B))\preceq I_*(A\otimes_K B).
$$
\end{Cor}

\subsection{Ore extensions}

In this section we study how the isoperimetric profile behaves in
Ore extensions. For the definition of an Ore extension we refer to
\cite{lenagan}.

\begin{Prop}
Let $A$ be an algebra, $\sigma$ an automorphism of $A$ and
$\delta$ a $\sigma$-derivation. If $I_*(A,V)$ is subadditive for
some subframe $V\subset A$, then we have
$$
I_*(A,V)\preceq I_*(A[x,\sigma,\delta],V+Vx).
$$
\end{Prop}
\begin{proof}
There is a natural filtration of $A[x,\sigma,\delta]$ determined
by the degree of $x$, such that the associated graded algebra is
isomorphic to $A[x,\sigma]$. Hence there is a valuation $\nu$ from
$A[x,\sigma,\delta]$ to $A[x,\sigma]$, which by Theorem
\ref{valuationinequality} gives
$$
I_*(A[x,\sigma,\delta],W)\succeq I_*(A[x,\sigma],W),
$$
for any graded subframe $W=\oplus_{i= 0}^{m}W_ix^i$.

Hence it's enough to show that $I_*(A[x,\sigma],V+Vx)\succeq
I_*(A,V)$, where $V$ is a subframe of $A$. First observe that the
leading-term map of $A[x,\sigma]$ is a valuation from
$A[x,\sigma]$ to itself. Again by Theorem
\ref{valuationinequality} it follows that it's enough to consider
only the graded subspaces of $A[x,\sigma]$.

Let $V$ be a subframe of $A$. Given a graded subspace $Z\subset
A[x,\sigma]$, we have $Z=\oplus_{i=0}^{n}Z_ix^i$, where
$Z_i\subset A$ for all $i$. Since $ax=x\sigma (a)$ for all $a\in
A$, we get
\begin{eqnarray*}
|\partial_{V+Vx}(Z)| & = &  \left|\sum_{i=0}^{n}
VZ_ix^i+VxZ_ix^i\right|-|Z| \\
 & = & \left|\sum_{i=0}^{n+1}
(VZ_i+VZ_{i-1}^{\sigma^{-1}})x^i\right|-\sum_{i=1}^{n}|Z_i|\\
& = & \sum_{i=0}^{n+1}\left|VZ_i+VZ_{i-1}^{\sigma^{-1}}\right|-\sum_{i=1}^{n}|Z_i|\\
 & \geq & \sum_{i=0}^{n}|VZ_i|-\sum_{i=1}^{n}|Z_i|= \sum_{i=0}^{n}
 |\partial_V(Z_i)|\\
  & \geq & \sum_{i=1}^{n}I_*(|Z_i|;A,V)\geq C_2
  I_*(C_1(\sum_{i=1}^{n}|Z_i|);A,V)=
  C_2I_*(C_1|Z|;A,V),
\end{eqnarray*}
where by convention $Z_{-1}=Z_{n+1}=\{0\}$, and $C_1$ and $C_2$
are the two positive constants coming from the subadditivity
assumption. This shows that
$$
I_*(A[x,\sigma],V+VX)\succeq I_*(A,V),
$$
completing the proof.
\end{proof}

The following corollary follows from the previous proposition and
Theorem \ref{domainsubadditive}.
\begin{Cor}
Let $A$ be a domain, $\sigma$ an automorphism of $A$ and $\delta$
a $\sigma$-derivation. If $A$ is amenable, then for any subframe
$V\subset A$,
$$
I_*(A,V)\preceq I_*(A[x,\sigma,\delta],V+Vx).
$$

If $A$ is nonamenable, then $A[x,\sigma,\delta]$ is nonamenable.
\end{Cor}

\begin{Rem}
Notice that in the proof of the previous proposition we used the
following obvious inequality
$$
\sum_{i=0}^{n+1}\left|VZ_i+VZ_{i-1}^{\sigma^{-1}}\right| \geq
\sum_{i=0}^{n}|VZ_i|.
$$

This inequality doesn't appear to be optimal and it's reasonable to expect a better one.

In this direction, in \cite{zhang}, Theorem 5.2, Zhang essentially
proves the following
\begin{Prop} \label{zhanglower}
Let $A$ be an algebra, $V\subset A$ a subframe, $\sigma$ an
automorphism of $A$ and $\delta$ a $\sigma$-derivation. If
$I_*(A,V)\succeq n^{\frac{d-1}{d}}$ for some $d\in \mathbb{R}$,
$d\geq 1$, then
$$
I_*(A[x,\sigma,\delta],V+Vx)\succeq n^{\frac{d}{d+1}}.
$$
\end{Prop}

This proposition gives for example a lower bound for the
isoperimetric profile of iterated Ore extensions, starting from a
finitely generated algebra $A$ with $I_*(A)\succeq
n^{\frac{d-1}{d}}$, for some $d\geq 1$.
\end{Rem}

We have also these two easy corollaries.
\begin{Cor}
Let $A$ be a finitely generated algebra and $\sigma$ an automorphism
of $A$, such that $\sigma^{m}$ is an inner automorphism for some
$m\in \mathbb{N}$. Then
$$
I_*(n;A[x,\sigma])\preceq I_*(n;A\otimes_KK[x]).
$$
\end{Cor}
\begin{proof}
If $\sigma^m$ is the inner automorphism given by the conjugation by
the invertible element $u\in A$, then $A[x,\sigma]$ is a finite free
module over $A[x^m,\sigma]\cong A[u^{-1}x]\cong A\otimes_K K[x]$.
The result now follows from Lemma \ref{finiterightmod}.
\end{proof}

There is also an analogous version of this corollary with the
algebra of Laurent skew polynomials $A[x,x^{-1},\sigma]$.

\begin{Cor}
Let $A$ be a finitely generated algebra and $\sigma$ an automorphism
of $A$, such that $\sigma^{m}$ is an inner automorphism for some
$m\in \mathbb{N}$. If $I_*(n;A)\sim n^{\frac{d-1}{d}}$ then
$$
I_*(n;A[x,\sigma])\sim n^{\frac{d}{d+1}}.
$$
\end{Cor}
\begin{proof}
It follows from the previous corollary, Corollary
\ref{tensorprodpolyn} and Proposition \ref{zhanglower}.
\end{proof}

\subsection{Modules and ideals}

If $V$ is a frame of a $K$-algebra $A$ and $M$ is a left $A$-module,
then we can define the isoperimetric profile of the $A$-module $M$
as
$$
I_*(n;M,V):=\inf |\partial_V(W)|= \inf |VW/W|
$$
where the infimum is taken over all $n$-dimensional subspaces $W$
of $M$. As for algebras, the asymptotic behavior of this function
does not depend on the generating subspace $V$, hence we can talk
about \textit{the isoperimetric profile of the module} $M$ and we
will denote it by $I_*(M)$. We observe some properties of this
isoperimetric profile.

\begin{Prop} \label{modules}
Let $A$ be an algebra, $V\subset A$ a subframe of $A$ and $M=
{}_AM$ a left $A$-module.
\begin{enumerate}
    \item [(i)]If $IM=0$ for some ideal $I$ of $A$, then $I_*({}_AM,V)\sim
I_*({}_{A/I}M,\overline{V})$, where $\overline{V}$ is the image of
$V$ in $A/I$.
    \item[(ii)] If $N$ is an $A$-submodule of $M$, then $I_*(M,V)\preceq
    I_*(N,V)$.
    \item[(iii)] If $M$ is a left $A$-module,
    then $I_*({}_AM,V)\preceq I_*(A,V)$.
\end{enumerate}
\end{Prop}
\begin{proof}
The first property follows directly from the definitions.

For (ii), given a subspace $W\subset N$, the boundary
$\partial_V(W)$ is the same as if we regard $W$ as a subspace of
$N$ or of $M$, hence $I_*(M,V)\preceq I_*(N,V)$.

Now by (ii), $I_*(M,V)\preceq I_*(Am,V)$ for all $m\in M$. Hence
we can assume that $M=Am$ for some $m\in M$. By (i) we can also
assume that $M$ is faithful. In this case, given a finite
dimensional subspace $W$ of $A$ we will have $|Wm|=|W|$. Then
clearly $|\partial_V(Wm)|\leq |\partial_V(W)|$. This gives the
inequality we wanted.
\end{proof}

Consider now a frame $V$ of an algebra $A$, and an infinite
dimensional ideal $J$ in $A$. Now $J$ is a left $A$-module, hence
$$
I_*(J)\preceq I_*(A).
$$
But also $J$ is an $A$-submodule of $A$, hence $I_*(A)\preceq
I_*(J)$. Therefore $I_*(A)\sim I_*(J)$ as $A$-modules.
\begin{Rem}
Notice that the isoperimetric profile of an ideal $J$ of an
algebra $A$ as an $A$-module is a priori different from the
isoperimetric profile of $J$ as a subalgebra of $A$.
\end{Rem}

\section{Computations of isoperimetric profiles of various algebras}

The aim of this section is to prove Theorem \ref{theorem5}, by
computing the isoperimetric profiles of the algebras listed in
there.

\subsection{Algebras of $GK$-dimension $1$}

For finitely generated algebras of $GK$-dimension 1 the
isoperimetric profile is constant.

\begin{Prop}
If $A$ is a finitely generated algebra of $GK$-dimension 1, then
$I_*(A)$ is constant.
\end{Prop}
\begin{proof}
Let $A$ be a finitely generated algebra of $GK$-dimension 1. G.
Bergman proved (see \cite{lenagan}, Theorem 2.5) that for an algebra
to have $GK$-dimension 1 is equivalent to have \textit{linear
growth}, i.e. if $V$ is a frame for $A$, then for all $n\in
\mathbb{N}$
$$
|V^{n+1}|-|V^n|\leq C,
$$
where $C$ is a positive constant. This inequality can also be
written as
$$
|\partial_V(V^n)|\leq C.
$$

Since the growth is linear, this proves that the isoperimetric
profile $I_*(A)$ is constant.
\end{proof}

\begin{Rem} \label{exendsalg}
The converse of this proposition is not true.

A cheap example is given by the algebra
$$
A=K[x]\oplus K[y,z].
$$

We know by Proposition \ref{directsum} that $I_*(A)\preceq
I_*(K[x])$, and we know by Proposition \ref{polynomialalg} that
$I_*(K[x])$ is constant. However, $GK\dim A=2$.

There is a more interesting example (cf. \cite{elek3}, Example 4).
Consider the algebra $A=K\langle x,y \rangle/J$, where $J$ is the
ideal generated by all monomials in $x$ and $y$ containing at least
2 $y$'s. Clearly $V=K+Kx+Ky$ is a frame of the infinite dimensional
algebra $A$. Observe that the numbers $a_n:=|V^n|$ satisfy the
relation $a_{n}=a_{n-1}+n$, with initial conditions $a_1=3$ and
$a_2=5$. Hence $A$ has quadratic growth, and $GK\dim A=2$. On the
other hand, if we put $W_n:=span_K\{y,xy,x^2y,\dots,x^{n-1}y\}$, we
have $|W_n|=n$, and
$$
|\partial_V(W_n)|=1
$$
for all $n\in \mathbb{N}$. This shows that $I_*(A)$ is constant.
\end{Rem}

Notice that both of these examples are not domains.
\begin{Ques}
Is it true that if a prime noetherian algebra has constant
isoperimetric profile, then it has $GK$-dimension 1?
\end{Ques}
Notice that the noetherianity assumption can't be dropped: the following example is due to Jason
Bell.
\begin{Ex}[J. Bell] \label{examplejbell}
Consider the algebra $A$ over $K$ with generators $x$ and $y$ and relations $x^2$, $xy^mx$ for $m$
not a power of $2$, and for each $r\geq 2$, $xy^{2^{m_1}}x y^{2^{m_2}}x\cdot\dots\cdot x
y^{2^{m_r}}x$ whenever $\sum_{i=1}^r m_i < r2^r$. This ring has $GK\dim$ $2$ and is prime. Let $V =
K + Kx+Ky$, and for $k\geq m+1$ let $W_k = span_K\{ y^ix \, :\, 2^k+1\leq i < 2^{k+1}\}$. Then
$xW_k =(0)$ and $yW_k + W_k = W_k+Ky^{2^{k+1}}x$. Hence $|VW_k/W_k| = 1$ and $|W_k|=2^k$. This
easily implies that the isoperimetric profile of $A$ is constant.
\end{Ex}

\subsection{Commutative Domains}

We compute the isoperimetric profile of finitely generated
commutative domains.
\begin{Prop} \label{commutative}
Let $A$ be a finitely generated commutative domain over $K$, and let
$d=GK\dim A$. Then $I_*(n;A)\sim n^{\frac{d-1}{d}}$.
\end{Prop}
\begin{proof}
By the Noether's normalization theorem the ring $A$ is a finitely
generated module over a subring $B$ isomorphic to
$K[x_1,\dots,x_d]$.

Theorem \ref{subOredomains} implies that
$$
n^{\frac{d-1}{d}}\sim I_*(B)\preceq I_*(A).
$$
Considering now the quotient fields $Q\subset S$ of $B$ and $A$
respectively, we have that $S$ is a finite dimensional vector space
over $Q$, hence using Lemma \ref{finiterightmod} and Corollary
\ref{localization} we have
$$
I_*(A)\preceq I_*(B),
$$
which gives the result.
\end{proof}

\subsection{$PI$ algebras}

We compute the isoperimetric profile of finitely generated prime
$PI$ algebras.
\begin{Prop}
If $A$ is a finitely generated prime $PI$ algebra, then $I_*(A)\sim
n^{\frac{d-1}{d}}$, where $d= GK\dim A$.
\end{Prop}
\begin{proof}
A theorem of Berele says that a finitely generated $PI$ algebra has
finite $GK$-dimension (see \cite{lenagan}, 10.7).

Suppose that $A$ is a finitely generated prime $PI$ algebra, and
consider its quotient algebra $Q$, which is known to be a full
matrix algebra over a division algebra $D$, which is a finite module
over its center $F$. Clearly $d=GK\dim F$, hence the result follows
from Proposition \ref{bigcorollary} (2).
\end{proof}

We have also the following
\begin{Cor}
If $A$ is a finitely generated semiprime $PI$ algebra, then
$I_*(A)\preceq n^{\frac{d-1}{d}}$, where $d= GK\dim A$.
\end{Cor}
\begin{proof}
The proof of this corollary goes like the one of the previous
proposition. In this case $Q$ is a direct sum of full matrix
algebras over division algebras, which are finitely generated over
their centers. Hence the same argument we used before together with
Proposition \ref{directsum} and well known properties of the
$GK$-dimension gives the result.
\end{proof}

Notice that in the semiprime case we have a direct sum of
subalgebras, hence Proposition \ref{directsum} shows that in
general we don't have the equivalence.

\subsection{Universal enveloping algebras}

We compute the isoperimetric profile of universal enveloping
algebras of finite dimensional Lie algebras.

\begin{Prop}
The isoperimetric profile of the universal enveloping algebra
$\mathcal{U}(\frak{g})$ of a finite dimensional Lie algebra
$\frak{g}$ is $I_*(n;\mathcal{U}(\frak{g}))\sim
n^{\frac{d-1}{d}}$, where $d=\dim \frak{g}$.
\end{Prop}
\begin{proof}
Theorem \ref{gradedassociated} applies to the universal enveloping
algebra $\mathcal{U}(\frak{g})$ of a finite dimensional Lie algebra
$\frak{g}$. Since $gr(\mathcal{U}(\frak{g}))$ (with respect to the
natural filtration) is isomorphic to the algebra of polynomials in
$d=\dim \frak{g}$ variables, we have the lower bound
$$
I_*(n;\mathcal{U}(\frak{g}))\succeq
I_*(n;gr(\mathcal{U}(\frak{g})))\sim n^{\frac{d-1}{d}}.
$$
Now consider a basis $e_1,e_2,\dots , e_d$ of $\frak{g}$, fix the
order $e_1<e_2<\dots <e_d$ and consider the lexicographical order on
the monomials in the $e_i$'s in $\mathcal{U}(\frak{g})$. For any
$n\in \mathbb{N}$ consider the subspace
$V_n=span_K\{e_{1}^{m_1}e_{2}^{m_2}\cdots e_{d}^{m_d}\mid \text{for
all $i$ }0\leq m_i\leq n-1\}$. If we call
$\mathcal{U}_1=span_K\{1,e_1,\dots ,e_d\}$, it follows from the
definition of $\mathcal{U}(\frak{g})$ and the PBW theorem that a
basis of the boundary $\partial_{\mathcal{U}_1}(V_n)$ is given by
the classes of the monomials $e_{1}^{k_1}e_{2}^{k_2}\cdots
e_{d}^{k_d}$ such that exactly one of the $k_i$'s is equal to $n$
and all the other are smaller then $n$. Now $|V_n|=n^d$ and
$|\partial_{\mathcal{U}_1}(V_n)|=dn^{d-1}= d|V_n|^{\frac{d-1}{d}}$.
From this follows easily the upper bound we needed.
\end{proof}

We want to derive also some consequences in the infinite
dimensional case.

\begin{Prop} \label{propliealgebras}
If $A=\mathcal{U}(\frak{g})$ is the universal enveloping algebra
of an infinite dimensional Lie algebra $\frak{g}$, then for any
$0<\alpha<1$ there exists a subframe $V\subset \mathcal{U}_1$ such
that
$$
I_*(n;\mathcal{U}(\frak{g}),V)\succneqq n^{\alpha}.
$$
\end{Prop}
\begin{proof}
A basis of $\mathcal{U}_1$ is given by a basis of $\frak{g}$ and
$1$. Now $gr(\mathcal{U}(\frak{g}))$ is isomorphic to the polynomial
algebra $K[x_1,x_2,\dots]$ on infinitely many variables, where each
variable $x_i$ corresponds to a basis element of $\frak{g}$.

Suppose first that $V=V_d\subset \mathcal{U}_1$, where a basis for
$V_d$ is given by the basis elements of $\mathcal{U}_1$
corresponding to $1,x_1,\dots,x_d$. Then by Theorem
\ref{gradedassociated}
$$
I_*(n,gr(A),V)\preceq I_*(n;A,V).
$$

But by Proposition \ref{freeleftmod}, since we can see $gr(A)\cong
K[x_1,x_2,\dots]$ as a free $K[V]\equiv K[x_1,\dots,x_d]$-module,
it follows that $I_*(n,gr(A),V)\succeq
I_*(n;K[x_1,\dots,x_d],V)\sim n^{\frac{d-1}{d}}$. It's easy to see
by considering the cubes in the $x_1,\dots,x_d$ as usual (and it
follows also from Proposition \ref{modules}) that
$I_*(n,gr(A),V)\preceq n^{\frac{d-1}{d}}$, and hence
$I_*(n,gr(A),V)\sim n^{\frac{d-1}{d}}$. From this the result
easily follows.
\end{proof}

This proposition implies for example that for a finitely generated
infinite dimensional Lie algebra (e.g. affine Kac-Moody algebras),
its universal enveloping algebras has an isoperimetric profile
faster then any polynomial in $n$ of degree $\alpha<1$.

\subsection{Weyl algebras}

Consider now the Weyl algebra $A_d=A_d(K)$, i.e. the algebra
$K\langle x_1,\dots,x_d,y_1,\dots, y_d\rangle$ subject to the
relations
$$
[x_i,x_j]=0= [y_i,y_j]\text{\quad and\quad}[x_i,y_j]=\delta_{i,j},
$$
where $\delta_{i,j}$ is the Kronecker symbol. It is well known
that $A_d$ is a domain.
\begin{Prop} \label{weylalg}
The isoperimetric profile of the Weyl algebra $A_d$ is
$$
I_*(n;A_d)\sim n^{\frac{2d-1}{2d}}.
$$
\end{Prop}
\begin{proof}
The lower bound $n^{\frac{2d-1}{2d}}\preceq I_*(n;A_d)$ is given
by Theorem \ref{gradedassociated}, since $gr(A_d)$ (with respect
to the filtration determined by total degree) is isomorphic to the
algebra of polynomials $K[ x_1,\dots,x_d,y_1,\dots, y_d]$.

Now for any $n\in \mathbb{N}$ consider the subspace
$V_n=span_K\{x_{1}^{m_1}\cdots x_{d}^{m_d}y_{1}^{m_{d+1}}\cdots
y_{d}^{m_{2d}}\mid \text{for all $i$ } 0\leq m_i\leq n-1\}$. It's
easy to see that a basis for $A_d$ is given by the monomials of
the form $x_{1}^{m_1}\cdots x_{d}^{m_d}y_{1}^{m_{d+1}}\cdots
y_{d}^{m_{2d}}$. Calling
$V=span_K\{x_1,\dots,x_d,y_1,\dots,y_d\}$, it's clear that a basis
for $\partial_V(V_n)$ is given by the classes of the monomials
$x_{1}^{k_1}\cdots x_{d}^{k_d}y_{1}^{k_{d+1}}\cdots
y_{d}^{k_{2d}}$ such that exactly one of the $k_i$'s is equal to
$n$  and all the other are smaller then $n$. Now $|V_n|=n^{2d}$
and $|\partial_V(V_n)|=2dn^{2d-1}=2d |V_n|^{\frac{2d-1}{2d}}$.
From this it follows easily the upper bound we needed.
\end{proof}

\subsection{Quantized algebras}

In this subsection we compute the isoperimetric profile of some
quantized algebras related to quantum groups.

We start with quantum skew polynomial algebras. Let $\{p_{ij}\mid
1\leq i<j\leq d\}$ be a set of nonzero scalars in $K$. The
\textit{quantum skew polynomial algebra}
$K_{p_{ij}}[x_1,\dots,x_d]$ is generated by the variables
$x_1,\dots,x_d$ subject to the relations $x_jx_i=p_{ij}x_ix_j$ for
all $i<j$. The set of ordered monomials $\{x_{1}^{l_1}\cdots
x_{d}^{l_d}\mid (l_1,\dots,l_d)\in \mathbb{N}^d\}$ is a basis over
$K$ of $K_{p_{ij}}[x_1,\dots,x_d]$. In \cite{zhang1}, Example 7.1,
Zhang gives a valuation from $K_{p_{ij}}[x_1,\dots,x_d]$ to
$K[x_1,\dots,x_d]$, hence by Theorem \ref{valuationinequality} we
have
$$
I_*(n;K_{p_{ij}}[x_1,\dots,x_d])\succeq
I_*(n;K[x_1,\dots,x_d])\sim n^{\frac{d-1}{d}}.
$$

Consider now the subspaces $V_n:=span_K\{x_{1}^{m_1}\cdots
x_{d}^{m_d}\mid \text{ for all $i$ }\,\,\, 0\leq m_i\leq n-1 \}$
corresponding to the cubes in $\mathbb{Z}_{\geq 0}^{d}$, and let
$V=span_K\{1,x_1,\dots ,x_d\}$. Clearly $|V_n|=n^d$, and from the
defining relations it follows that $|\partial_V(V_n)|=dn^{d-1}=d
|V_n|^{\frac{d-1}{d}}$ (see the proof of Corollary \ref{weylalg}).
From this it follows easily the upper bound
$$
I_*(n;K_{p_{ij}}[x_1,\dots,x_d])\preceq
I_*(n;K[x_1,\dots,x_d])\sim n^{\frac{d-1}{d}},
$$
giving
$$
I_*(n;K_{p_{ij}}[x_1,\dots,x_d])\sim n^{\frac{d-1}{d}}.
$$

The following definition is in \cite{zhang1}, Section 7.

\begin{Def}
Consider the lexicographical order on $\mathbb{Z}^d$ with
$\mathrm{deg}(e_i)<\mathrm{deg}(e_j)$ for $i<j$, where $e_i$ is
the vector with $1$ in the $i$-th position, and $0$ elsewhere. An
algebra $A$ is called a \textit{filtered skew polynomial algebra
in $d$ variables} if there is a set of generators
$\{x_1,\dots,x_d\}$ of $A$ such that the following three
conditions hold.
\begin{itemize}
    \item[(q1)] The set of monomials $\{x_{1}^{l_1}\cdots x_{d}^{l_d}\mid (l_1,\dots,l_d)\in
\mathbb{N}^d\}$ is a basis over $K$ of $A$. We define
$\mathrm{deg}(x_{1}^{l_1}\cdots x_{d}^{l_d})=(l_1,\dots,l_d)$ and
$F_{(l_1,\dots,l_d)}$ to be the set of all linear combinations of
monomials of degree $\leq (l_1,\dots,l_d)$.
    \item[(q2)] $\{F_{(l_1,\dots,l_d)}\mid (l_1,\dots,l_d)\in
    \mathbb{N}^d\}$ is a filtration of $A$.
    \item[(q3)] The associated graded algebra $gr(A)$ is
    isomorphic to a quantum skew polynomial algebra.
\end{itemize}
\end{Def}

For example it's easy to see that the Weyl algebras are filtered
skew polynomial algebras.

The following proposition is an immediate consequence of Theorem
\ref{gradedassociated} and what we have shown before.
\begin{Prop} \label{filteredskew}
If $A$ is a filtered skew polynomial algebra in $d$ variables,
then
$$
I_*(n;A)\succeq n^{\frac{d-1}{d}}.
$$
\end{Prop}

Now we want to consider the quantum matrix algebras
$M_{q,p_{ij}}(d)$ and the quantum groups $GL_{q,p_{ij}}(d)$. See
\cite{artin} for details on these algebras.

Given a set of nonzero scalars $\{q\}\cup \{p_{ij}\mid 1\leq i<j\leq d\}$, the \textit{quantum matrix algebra} $M_{q,p_{ij}}(d)$ is generated by
$\{x_{ij}\mid 1\leq i,j\leq d\}$ subject to the relations (7.4.1) of \cite[p. 2885]{zhang1}. It's easy to show (cf. \cite{zhang1}, Example 7.4)
that $M_{q,p_{ij}}(d)$ is a filtered skew polynomial algebra on $d^2$ variables, hence by Proposition \ref{filteredskew}
$$
I_*(n;M_{q,p_{ij}}(d))\succeq n^{\frac{d^2-1}{d^2}}.
$$

To prove the other inequality, for each $n\in \mathbb{N}$ we
define the subspace
$$
V_n:=span_K\{x_{11}^{m_{11}} x_{12}^{m_{12}}\cdots x_{1d}^{m_{1d}}
x_{21}^{m_{21}} \cdots x_{2d}^{m_{2d}}\cdots x_{dd}^{m_{dd}} \mid
\text{for all $i$ and $j$ }\,\,\, 0\leq m_{ij}\leq n-1\},
$$
and we put $V:=K+ span_K\{x_{ij}\mid 1\leq i,j\leq d\}$. Using the defining relations it's easy to show that $VV_n\subset V_{n+1}$. This would
imply that
\begin{eqnarray*}
|\partial_V(V_n)| & = & |VV_n|-|V_n|\leq
|V_{n+1}|-|V_n|\\
 & = & (n+1)^{d^2}-n^{d^2}\sim
n^{d^2-1}=|V_n|^{\frac{d^2-1}{d^2}}.
\end{eqnarray*}
As usual, from this it follows easily the upper bound
$$
I_*(n;M_{q,p_{ij}}(d))\preceq n^{\frac{d^2-1}{d^2}},
$$
which gives
$$
I_*(n;M_{q,p_{ij}}(d))\sim n^{\frac{d^2-1}{d^2}}.
$$

The \textit{quantum group} $GL_{q,p_{ij}}(d)$ is defined to be the
localization $M_{q,p_{ij}}(d)[D^{-1}]$, where $D$ is the quantum
determinant of $M_{q,p_{ij}}(d)$, and $M_{q,p_{ij}}(d)[D^{-1}]$
indicates the right localization with respect to the subset
$\{D^n\mid n\in \mathbb{N}\}$. Hence by Corollary \ref{localization}
we have
$$
I_*(n;GL_{q,p_{ij}}(d))\sim I_*(n;M_{q,p_{ij}}(d))\sim
n^{\frac{d^2-1}{d^2}}.
$$

Consider now the quantum Weyl algebra $A_d(q,p_{ij})$ (see
\cite{giaquinto} for details on this algebras).

Given a set of nonzero scalars $\{q\}\cup \{p_{ij}\mid 1\leq i<j\leq
d\}$, the \textit{quantum Weyl algebra} $A_d(q,p_{ij})$ is generated
by $\{x_1,\dots,x_d,y_1,\dots ,y_d\}$ subject to the relations given
in \cite{zhang1}, Example 7.5. It's easy to see (cf. \cite{zhang1},
Example 7.5) that defining $\mathrm{deg}(x_i)=d+1-i$ and
$\mathrm{deg}(y_i)=2d+1-i$, $A_d(q,p_{ij})$ is a filtered skew
polynomial algebra in $2d$ variables. Hence by Proposition
\ref{filteredskew} we have
$$
I_*(n;A_d(q,p_{ij}))\succeq n^{\frac{2d-1}{2d}}.
$$

To prove the other inequality, for each $n\in \mathbb{N}$ we
define the subspace
$$
V_n:=span_K\{x_{1}^{m_{1}} \cdots x_{d}^{m_{d}} y_{1}^{n_{1}}
\cdots y_{d}^{n_{d}} \mid \text{for all $i$ and $j$ }\,\,\, 0\leq
m_{i},n_j\leq n-1\},
$$
and we put $V:=K+ span_K\{x_1,\dots,x_d,y_1,\dots,y_d\}$. Again we can show that $VV_n\subset V_{n+1}$, from which it follows easily the upper
bound
$$
I_*(n;A_d(q,p_{ij}))\preceq n^{\frac{2d-1}{2d}},
$$
which gives
$$
I_*(n;A_d(q,p_{ij}))\sim n^{\frac{2d-1}{2d}}.
$$

Consider now the \textit{quantum group} $\mathcal{U}(\frak{sl}_2)$
(see \cite{jimbo}). This is an algebra isomorphic to an algebra
generated by $\{e,f',h\}$ subject to the relations (7.6.2) of
\cite{zhang1}, pag. 2887.

It's easy to see that it is a filtered skew polynomial algebra in
three variables, setting $\mathrm{deg}(h)=(1,0,0)$,
$\mathrm{deg}(e)=(0,1,0)$ and $\mathrm{deg}(f')=(0,0,1)$ (cf.
\cite{zhang1}, Example 7.6). This by Proposition
(\ref{filteredskew}) gives the lower bound
$$
I_*(n;\mathcal{U}(\frak{sl}_2))\succeq n^{\frac{2}{3}}.
$$

Now consider for each $n\in \mathbb{N}$ the subspace
$$
V_n:=span_K\{h^{m_1}e^{m_2}f'^{m_3}\mid 0\leq m_1\leq 2(n-1)\,
\text{ and }\,\, 0\leq m_i\leq n-1\,\, \text{ for $i=2,3$}\},
$$
and let $V=span_K\{1,h,e,f'\}$. We can show that $VV_n\subseteq
V_{n+1}$, from which it follows easily the upper bound
$$
I_*(n;\mathcal{U}(\frak{sl}_2))\preceq n^{\frac{2}{3}},
$$
which gives
$$
I_*(n;\mathcal{U}(\frak{sl}_2))\sim n^{\frac{2}{3}}.
$$

There is also another version of the quantum universal enveloping
algebra $\mathcal{U}(\frak{sl}_2)$, say $\mathcal{U}'(\frak{sl}_2)$,
which was studied in \cite{jing}. Given $q\in K\setminus \{0\}$, the
quantum universal enveloping algebra $\mathcal{U}'(\frak{sl}_2)$ is
generated by $\{e,f,h\}$ subject to the relations
\begin{eqnarray} \label{relationsquantumsl2zhang}
\notag qhe-eh & = & 2e,\\
hf-qfh & = & -2f,\\
\notag ef-qfe & = & h+ \frac{1-q}{4}h^{2}.
\end{eqnarray}
Defining $\mathrm{deg}(h)=(1,0,0)$, $\mathrm{deg}(e)=(0,1,0)$ and
$\mathrm{deg}(f)=(0,0,1)$, $\mathcal{U}'(\frak{sl}_2)$ is a
filtered skew polynomial algebra in three variables (cf.
\cite{zhang1}, Example 7.6). This by Proposition
\ref{filteredskew} gives the lower bound
$$
I_*(n;\mathcal{U}'(\frak{sl}_2))\succeq n^{\frac{2}{3}}.
$$

For the upper bound we can use the same subspaces $V_n$ (where of
course we replace $f'$ with $f$).

We summarize the computations of this section in the following
\begin{Prop}
With the notations we explained in this subsection,
\begin{itemize}
    \item[(1)] $I_*(n;K_{p_{ij}}[x_1,\dots,x_d])\sim n^{\frac{d-1}{d}}$;
    \item[(2)] $I_*(n;M_{q,p_{ij}}(d))\sim n^{\frac{d^2-1}{d^2}}$;
    \item[(3)] $I_*(n;GL_{q,p_{ij}}(d))\sim n^{\frac{d^2-1}{d^2}}$;
    \item[(4)] $I_*(n;A_d(q,p_{ij}))\sim n^{\frac{2d-1}{2d}}$;
    \item[(5)] $I_*(n;\mathcal{U}(\frak{sl}_2))\sim
    n^{\frac{2}{3}}$;
    \item[(6)] $I_*(n;\mathcal{U}'(\frak{sl}_2))\sim
    n^{\frac{2}{3}}$.
\end{itemize}
\end{Prop}

All together the computations we performed in this section give a
proof of Theorem \ref{theorem5}.

\section{Relations with other invariants}

In this section we compare the isoperimetric profile to some other
invariants for infinite dimensional algebras.

\subsection{$I_*$ and the F\o lner function}

Given an amenable algebra $A$ and a subframe $V$ of $A$, we define
the \textit{F\o lner function} $F_*(n;A,V)$ \textit{with respect
to} $V$ (cf. \cite{gromov}) to be the minimal dimension of a
subspace $W$ of $A$ such that
$$
|\partial_V(W)|\leq \frac{|W|}{n}.
$$

Notice that this function is not defined for a nonamenable
algebra.

As we did for the isoperimetric profile, we say that an algebra
$A$ has F\o lner function if there exists a subframe $V$ of $A$
such that
$$
F_*(A,W)\preceq F_*(A,V)
$$
for any subframe $W$ of $A$. We denote this function and its
asymptotic equivalence class by $F_*(A)$, and we say that a
subframe $V$ measures $F_*(A)$ if $F_*(A)\sim F_*(A,V)$.

It can be proved in the same way as we did for the isoperimetric
profile that a finitely generated algebra $A$ has F\o lner function,
and its asymptotic behavior is measured by any frame $V$ of $A$.

Notice that if $n$ is in the image of $F_*(A,V)$, then
$$
I_*(n)=I_*(|W|)\leq |\partial_V(W)|\leq
\frac{|W|}{F_{*}^{-1}(|W|)}
$$
for a suitable subspace $W$ of dimension $n$. This would suggest
the inequality
$$
I_*(n)\preceq \frac{n}{F_{*}^{-1}(n)},
$$

where $F_{*}^{-1}(n):=\sup \{k\mid F_*(k)\leq n\}$.
\begin{Ques}
Is this inequality always true? Is it true for domains? Is it true
for semigroups?
\end{Ques}

Of course there is the analogous definition for semigroups: in
this case the F\o lner function is denoted by $F_{\circ}$ (cf.
\cite{gromov}).

In \cite{gromov} there are various proofs of the lower bound for the
F\o lner function of $\mathbb{Z}_{\geq 0}^{d}$, the upper bound
being clear considering the cubes:
$$
F_{\circ}(n;\mathbb{Z}_{\geq 0}^{d})\sim n^d.
$$

Notice that in this particular case $I_{\circ}(n)\sim n/
F_{\circ}^{-1}(n)$.
\begin{Ques}
Are these two functions always equivalent? Is it true for algebras?
Is it true for domains?
\end{Ques}

The equivalence $I_{*}(n)\sim n/ F_{*}^{-1}(n)$ is correct at least
in the case of polynomial algebras. In fact, using the fact that the
F\o lner functions of an orderable semigroup and its semigroup
algebra are asymptotically equivalent (see \cite{gromov}, Section
3), we have
$$
F_*(n;K[x_1,\dots,x_d])\sim n^d.
$$

Sometimes the F\o lner function is easier to handle than the
isoperimetric profile (see \cite{erschler}). For example the F\o
lner function of the tensor products has an easier relation with the
F\o lner functions of the factors.
\begin{Prop}
Given $A$ and $B$ two $K$-algebras, if $V_A$ and $V_B$ are two
subframes of $A$ and $B$ respectively, and $V:=V_A\otimes 1+
1\otimes V_B$, we have
$$
F_*\left(\frac{mn}{m+n};A\otimes_KB,V\right)\leq
F_*(m;A,V_A)F_*(n;B,V_B).
$$
\end{Prop}
\begin{proof}
We use the proof of Proposition \ref{tensorprodprofile}: we keep
the same notation we used there, but this time we choose suitable
subspaces $W\subset A$ and $Z\subset B$ for which
$|\partial_{V_A}(W)|\leq |W|/m$ and $|\partial_{V_B}(Z)|\leq
|Z|/n$. We get
\begin{eqnarray*}
|\partial_V(W\otimes Z)| & \leq &  |Z||\partial_{V_A}(W)|+|W|
|\partial_{V_B}(Z)|\\
 & \leq &  \frac{|W||Z|}{m}+\frac{|W||Z|}{n}=
\frac{m+n}{mn}|W||Z|,
\end{eqnarray*}
which gives the result.
\end{proof}

Putting $m=n$ in the proposition we get the following
\begin{Cor}
In the same notation of the previous proposition,
$$
F_*\left(n;A\otimes_KB,V\right)\preceq F_*(n;A,V_A)F_*(n;B,V_B).
$$
\end{Cor}

\subsection{$I_*$ and the lower transcendence degree}

In \cite{zhang} J. J. Zhang introduced the notion of the lower
transcendence degree of an algebra.
\begin{Def}
If for every subframe $V\subset A$ there is a subspace $W\subset
A$ such that
$$
|\partial_V(W)|=0,
$$
then we define the \textit{lower transcendence degree} of $A$ to
be $0$ and we write $\mathrm{Ld}(A)=0$. Otherwise there is a
subframe $V$ such that for every subspace $W$
$$
|\partial_V(W)|\geq 1.
$$
In this case the \textit{lower transcendence degree} of $A$ is
defined to be
$$
\mathrm{Ld}(A):=\sup_V \,\,\sup \{d\in \mathbb{R}_{\geq 0}\mid
\exists\,\, C>0\,\, :\,\, |\partial_V(W)|\geq
C|W|^{1-\frac{1}{d}}\text{ for all $W$}\},
$$
where $V$ ranges over all subframes of $A$. Hence $\mathrm{Ld}(A)$
is a nonnegative real number or infinity.
\end{Def}

Observe that in the definition of the lower transcendence degree we
can use the inequality $I_*(|W|;A,V)\succeq |W|^{1-\frac{1}{d}}$
instead of $|\partial_V(W)|\geq C|W|^{1-\frac{1}{d}}$. In the case
of a finitely generated algebra, since we already showed that the
asymptotic behavior of the isoperimetric profile does not depend on
the frame, we can drop the first supremum in the definition and we
can take simply some fixed frame $V$.

It's now clear from the definitions that if two algebras $A$ and
$B$ satisfy $I_*(A)\sim I_*(B)$, then
$\mathrm{Ld}(A)=\mathrm{Ld}(B)$. The converse is not always true:
\begin{Rem}
In general we do not have the inequality
\begin{equation} \label{lowerdegree}
n^{1-\frac{1}{\mathrm{Ld}(A)}} \preceq I_*(n).
\end{equation}
For example in the case $I_*(n)\sim n^{\alpha}/ \log n$ for some
$0<\alpha\leq 1$, we would have
$$
n^{\beta} \precneqq     I_*(n)
$$
for any $\beta <\alpha$, but
$$
n^{\gamma} \succneqq  I_*(n)
$$
for any $\gamma\geq \alpha$. For example, $I_*(n)\sim n/ \log n$
($\alpha=1$) is the isoperimetric profile of the group algebra of
a finitely generated polycyclic group of exponential growth (see
\cite{pittet}). Hence $(\mathrm{Ld}(A)-1)/\mathrm{Ld}(A)=\alpha$
in this case, which shows that the inequality is not true.
\end{Rem}

From this remark we see that if we have for example two algebras
$A$ and $B$ with $I_*(n;A)\sim n/ \log n$ and $I_*(n;B)\sim n$
(e.g. the group algebra of a finitely generated polycyclic group
of exponential growth and a free algebra of rank two), then
clearly $I_*(A) \nsim I_*(B)$, but
$\mathrm{Ld}(A)=\mathrm{Ld}(B)=\infty$. All this shows that the
isoperimetric profile is finer than the lower transcendence degree
as an invariant for algebras.

The following proposition follows directly from the definitions
\begin{Prop} \label{osservazhang}
If $d=\mathrm{Ld}(A)$, then $n^{\frac{s-1}{s}}\precneqq
I_*(n;A,V)$ for any $s\lneqq d$ and some particular subframe
$V\subset A$. Moreover, $I_*(n;A,W)\nsucceq n^{\frac{t-1}{t}}$ for
any $t>d$ and any subframe $W\subset A$.
\end{Prop}

In \cite{zhang}, Proposition 1.4, Zhang proves that for any
algebra $A$,
$$
\mathrm{Ld}A\leq \mathrm{Tdeg}A\leq GK\dim A,
$$
where $\mathrm{Tdeg}A$ is the Gelfand-Kirillov transcendence
degree (see \cite{zhang} for the definition).

This together with Proposition \ref{osservazhang} implies the
following theorem, which generalizes a result in \cite{elek2}.
\begin{Thm} \label{theorem6}
If all the finitely generated subalgebras of an algebra $A$ have
finite lower transcendence degree, then $A$ is amenable.
\end{Thm}

An example of a finitely generated amenable division algebra with
infinite $GK$-transcendence degree is given in \cite{elek2}. Theorem
\ref{theorem6} together with previous results in this paper allows
us to provide new examples of this sort.

An easy example is the field $F:=K(x_1,x_2,\dots)$ of rational
functions in infinitely many variables.

Even more interesting examples come from universal enveloping
algebras of infinite dimensional Lie algebras with subexponential
growth, for example affine Kac-Moody algebras. In fact by
\cite{smith} these algebras have subexponential growth, and so
they are amenable (see \cite{elek1}). But from Proposition
\ref{propliealgebras} it follows that they have infinite lower
transcendence degree. Since they are domains, we can consider
their quotient division algebras to provide examples of division
algebras.

In \cite[p. 181]{zhang}, Zhang asked if is it true that for any orderable semigroup $\Gamma$ the semigroup algebra $K\Gamma$ is
$\mathrm{Ld}$-stable, i.e. $\mathrm{Ld}K\Gamma=GK\dim K\Gamma$. We conclude the subsection giving a positive answer:
\begin{Prop}
The group algebra $K\Gamma$ of an ordered semigroup $\Gamma$ is
$\mathrm{Ld}$-stable.
\end{Prop}
\begin{proof}
By a theorem of Gromov (see \cite{gromov}, Section 3) we know that
$I_{\circ}(\Gamma,S)\sim I_*(K\Gamma,S)$ for any finite subset
$S\subset \Gamma$. Observe that $d:=GK\dim K\Gamma$ is the degree of
growth of the semigroup $\Gamma$, which may be of course infinity.
Now by the Couhlon-Saloff-Coste inequality (Theorem
\ref{coulsaloff}) we have
$$
I_*(n;\Gamma)\succeq n^{\frac{d-1}{d}},
$$
in case $d$ is finite, or
$$
I_*(n;\Gamma)\succeq n/\Phi (n),
$$
where $\Phi$ is the inverse function of the growth of $\Gamma$, if
$d$ is infinity. In the last case $\Phi$ is slower then any
positive power of $n$, hence in both cases
$$
\mathrm{Ld}K\Gamma\geq GK\dim K\Gamma.
$$

Since the other inequality is always true, this completes the
proof.
\end{proof}

\subsection{$I_*$ and the growth}

The Weyl algebra $A_1$ and its quotient division algebra $D_1$
give an example that shows that the isoperimetric profile is not a
finer invariant then the $GK$-dimension. Another example is in
\cite{lenagan}, Example 4.10, where the algebra
$\mathcal{U}(\frak{g})$ and some its localization have different
$GK$-dimensions, but they have the same isoperimetric profiles.

We may ask for an analogue of the Coulhon-Saloff-Coste inequality
(Theorem \ref{coulsaloff}) for algebras. In Remark \ref{exendsalg}
we considered the algebra $A=K\langle x,y \rangle/J$, where $J$ is
the ideal generated by all monomials in $x$ and $y$ containing at
least 2 $y$'s. We already showed that this algebra has constant
isoperimetric profile, but it has $GK$-dimension 2. This example
shows that we don't have in general an analogue for algebras of
the Coulhon-Saloff-Coste inequality. A cheaper example of this
type is the algebra $K[x]\oplus K\langle y,z \rangle$, which we
also considered in the Remark \ref{exendsalg}. Both these examples
are not domains.

An example of a prime algebra is Example \ref{examplejbell}. An
example of a domain is given by the quotient division algebra
$D_1$ of the Weyl algebra $A_1$.

In \cite{gromov}, Section 1.9, Gromov asks if there is a bound on
the growth of a domain by its F\o lner function. Keeping in mind
Questions 4 and 5, this bound would correspond to the
Coulhon-Saloff-Coste inequality for the isoperimetric profile. The
algebra $D_1$ answers this question in the negative, since in this
case clearly the F\o lner function $F_*(n)$ of $D_1$ is
asymptotically bounded by $n^2$, but $D_1$ grows exponentially. Of
course $D_1$ is not finitely generated.

A finitely generated example is given by the localization
$A_1\Omega^{-1}$ of the multiplicative closed subset $\Omega$ (of
the Weyl algebra $A_1$) generated by $x$ and $y$. This is a
finitely generated noetherian domain with $GK$-dimension $3$ but
with lower transcendence degree $2$ (see Example 4.11 in
\cite{lenagan} for details).

\section*{Acknowledgements}

I would like to thank my advisor Prof. Efim Zelmanov for his
invaluable encouragement and for providing me the proof of
Proposition \ref{zelmanov}. Thanks to Prof. Daniel Rogalsky for
having pointed me the important reference \cite{zhang} and for
useful discussions. Thanks also to Jason Bell for providing and
explaining to me Example \ref{examplejbell}, to Professors Tullio
Ceccherini-Silberstein and Aryeh Samet-Valliant for providing me
their preprint \cite{tullio}, and  to Prof. Anna Erschler for having
pointed me the reference \cite{pittet}.

\end{document}